\documentclass{amsart}
\usepackage[a4paper]{geometry}
\usepackage{csquotes}
\usepackage{afterpage}
\usepackage[T1]{fontenc}
\usepackage{kpfonts}
\usepackage{lmodern}
\usepackage{amscd,indentfirst,epsfig}
\usepackage{times}
\usepackage{enumerate}
\usepackage{mathrsfs}
\usepackage{stmaryrd}
\usepackage{mathalfa}
\usepackage{amsopn}
\usepackage{empheq,bm, bbm}
\usepackage{amsbsy,amsmath}
\usepackage{amscd}
\usepackage{comment}
\usepackage{appendix}
\usepackage{caption}
\usepackage{microtype}
\usepackage{mathtools, amsthm, amsfonts, amssymb, stmaryrd, dsfont}
\usepackage{caption}
\usepackage{subcaption}
\usepackage{url}
\usepackage{doi}
\usepackage{csquotes}
\usepackage{colortbl}
\usepackage{ esint }
\usepackage{arydshln}
\usepackage{tabularx}
\usepackage{chngcntr}
\usepackage{etoolbox}
\usepackage{array}
\usepackage{tikz}
\usepackage[most]{tcolorbox}
\usepackage{xcolor}
\usepackage{listings}
\usepackage[stable]{footmisc}
\usetikzlibrary{shadows}

\usepackage{cancel} 

\tcbset{colback=yellow!10!white, colframe=red!50!black, 
        highlight math style= {enhanced, 
            colframe=red,colback=red!10!white,boxsep=0pt}
        }

\newtcolorbox{mybox}[1]{colback=blue!5!white,colframe=blue!50,fonttitle=\bfseries,title=#1}

\definecolor{myblue}{rgb}{.9, .9, 1}

\usetikzlibrary{arrows}
\usetikzlibrary{shapes.geometric,decorations.markings,decorations.pathmorphing,decorations.pathreplacing,calligraphy}
\usetikzlibrary{patterns}
\usetikzlibrary{shadows}

\tikzset{
diagonal fill/.style 2 args={fill=#2, path picture={
\fill[#1, sharp corners] (path picture bounding box.south west) -|
                         (path picture bounding box.north east) -- cycle;}},
reversed diagonal fill/.style 2 args={fill=#2, path picture={
\fill[#1, sharp corners] (path picture bounding box.north west) |- 
                         (path picture bounding box.south east) -- cycle;}}
}

\newcommand{\R}{\mathbb{R}}
\newcommand{\N}{\mathbb{N}}

\newcommand{\ScEll}{{\mathcal{S}^+_\Ell}}
\newcommand{\ScEllone}{{\mathcal{S}^{+,1}_\Ell}}
\newcommand{\Ell}{\mathscr{L}}
\newcommand{\B}{\mathcal{B}}

\newcommand{\E}{\mathcal{E}}
\newcommand{\C}{\mathcal{C}}

\newcommand{\HH}{\mathcal{H}}

\newcommand{\dd}{\mathrm{d}}
\newcommand{\Hf}{\mathfrak{H}}
\newcommand{\perTr}{\underline{\mathrm{Tr}}}

\newcommand{\indic}{\mathbf{1}\hspace{-0.27em}\mathrm{l}}

\newcommand{\husimi}{\widetilde{W}_\Ell}

\newcommand{\coh}[2]{| #1, #2 \rangle}

\newcommand{\percoh}[2]{| \underline{#1, #2} \rangle}

\newcommand{\proj}[1]{| #1 \rangle \langle #1 |}

\newcommand{\e}{\varepsilon}

\renewcommand{\leq}{\leqslant}
\renewcommand{\geq}{\geqslant}

\newtheorem{proposition}{Proposition}
\newtheorem{theorem}{Theorem}
\newtheorem*{theorem*}{Theorem}
\newtheorem{definition}{Definition}

\newtheorem{lemma}{Lemma}

\newtheorem{remark}{Remark}
\newtheorem{corollary}{Corollary}

\title[Observability von Neumann crystal]{Observability inequality for the von Neumann equation in crystals}

\author[T. Borsoni]{T. Borsoni}
\address{T.B.: CERMICS, École des Ponts et Chaussées - Institut Polytechnique de Paris \& Inria}
\email{thomas.borsoni@enpc.fr}

\author[V. Ehrlacher]{V. Ehrlacher}
\address{V.E.: CERMICS, École des Ponts et Chaussées - Institut Polytechnique de Paris \& Inria}
\email{virginie.ehrlacher@enpc.fr}

\date{\today}


\begin{document}

\begin{abstract}
    We provide a quantitative observability inequality for the von Neumann equation on $\R^d$ in the crystal setting, uniform in small $\hbar$. Following the method of~\cite[F. Golse and T. Paul, 2022]{golse2022quantitative} proving this result in the non-crystal setting, the method relies on a stability argument between the quantum (von Neumann) and classical (Liouville) dynamics and uses an optimal transport-like pseudo-distance between quantum and classical densities. Our contribution yields in the adaptation of all the required tools to the periodic setting, relying on the Bloch decomposition, notions of periodic Schrödinger coherent state, periodic Töplitz operator and periodic Husimi densities.
\end{abstract}

\maketitle

\tableofcontents

\addtocontents{toc}{\protect\setcounter{tocdepth}{-1}} 

\section*{Introduction}

\addtocontents{toc}{\protect\setcounter{tocdepth}{1}} 
We obtain a quantitative observability equality for the von Neumann equation on $\R^d$ for crystals, uniform in the semi-classical limit of small $\hbar$. We use a method developed in~\cite{golse2022quantitative}, based on a classical/quantum stability argument with an optimal transport-like pseudo-metric between classical probability densities and quantum density matrices. Our contribution with respect to~\cite{golse2022quantitative} lies in the adaptation of the tools and arguments to the periodic (crystal) case. As such, we heavily rely on the Bloch decomposition, and consider periodic-adapted objects: the periodic trace (Definition~\ref{def:periodicTrace}), periodic Schrödinger coherent states, Töplitz and Husimi transforms (Definitions~\ref{def:periodizedCoherent},~\ref{def:periodizedToplitz} and~\ref{def:periodizedHusimi}) of which we provide various properties (Section~\ref{sec:propertiesSchroTopHus}), and a periodic variant of the optimal transport-like pseudo-metric (Definitions~\ref{def:class-quandcoupl},~\ref{def:class-quandcost} and~\ref{def:class-quandpseudometric}).
We mention the slight tweak of the cost with~\eqref{eqdef:Theta} which allows to make it work.

\medskip

This work stems from the study of the controllability of the Schrödinger equation, which has been a broad and active field~\cite{zuazua2002remarks,laurent2013survey}. We highlight the result of~\cite{lebeau1992controle,bardos1992sharp}, ensuring that the Geometric Control condition~\eqref{eq:GC} at a given time is sufficient for the exact controllability of the Schrödinger equation at any time, condition on which we rely here. There are also lines of work that do not rely on this condition (for instance on squared and perforated domains~\cite{jaffard1990controle},~\cite{burq1993controle} and the remark in~\cite[Equation~(2.1)]{zuazua2002remarks}, or based on uncertainty principle~\cite[Theorem 1.9]{d2025dynamical}). The controllability problem for a given system can often be addressed by turning it into an observability problem over an adjoint system through the Hilbert Uniqueness Method (HUM)~\cite{lions1988exact}. As pointed out in~\cite{zuazua2002remarks}, the adjoint system for the Schrödinger equation is the equation itself; as such, for this problem, controllability and observability are equivalent. Consider some open $\omega \subset \R^d$ and $\Psi^{\rm in} \subset L^2(\R^d)$, we say that an observability inequality holds for the Schrödinger equation over $\omega$, $\Psi^{\rm in}$ and some final time $T>0$ if there is a $C \equiv C(\omega,\Psi^{\rm in},T)>0$ such that for any $\psi$ solution to the Schrödinger equation with initial datum in $\Psi^{\rm in}$ we have
\begin{equation} \label{eqintro:observabilityschro}
\int_0^T \int_\omega |\psi(t,x)|^2 \, \dd t\, \dd x \geq C \|\psi(0)\|^2_{L^2}.
\end{equation}
Exact controllability (with respect to $\omega$ and $T$) then holds as soon as the above is true with $\Psi^{\rm in} = L^2(\R^d)$ (see~\cite{zuazua2002remarks}). 

\medskip

In this work, we rather focus on the von Neumann equation, a generalization of the Schrödinger equation to mixed states -- trace-class nonnegative self-adjoint operators over $L^2(\R^d)$. In this setting,~\eqref{eqintro:observabilityschro} writes
\begin{equation} \label{eqintro:observabilityvonNeu}
\int_0^T \text{trace} \left( \mathds{1}_\omega \, R(t) \, \mathds{1}_\omega  \right) \, \text{d} t \geq C \, \text{trace} \left(R(0)\right),
\end{equation}
where for all $t\in [0,T]$, $R(t)$ denotes a bounded self-adjoint non-negative trace-class operator acting on $L^2(\mathbb{R}^d)$, which is a suitable solution to the von Neumann equation. Our system of interest in the present work being an infinite periodic crystal in $\R^d$, it is relevant to turn to the von Neumann equation instead of the Schrödinger equation: the notion of wavefunction is no longer justified there (as the number of electrons in the system is infinite) but the notion of density matrix still is.

\medskip

In the next Section~\ref{sec:mainresult}, we introduce the mathematical framework of our studied crystal, define the suitable notions of periodic density matrices, trace, Schrödinger coherent state, Töplitz and Husimi transforms, and state our main result (Theorem~\ref{theo:obsfinal}), an observability inequality for the von Neumann equation in the crystal setting. Then, following the method of~\cite{golse2022quantitative}, we establish in Section~\ref{sec:stability} a stability property between the (classical) Liouville and (quantum) von Neumann dynamics, relying on an (periodically adapted) optimal transport-like pseudo-metric. Section~\ref{sec:propertiesSchroTopHus} is devoted to the obtention of useful properties of the periodic Schrödinger coherent states, Töplitz and Husimi transforms. Combining the results of Sections~\ref{sec:stability} and~\ref{sec:propertiesSchroTopHus}, we conclude in Section~\ref{sec:proof} to Theorem~\ref{theo:obsfinal}. The appendix contains tools on the Bloch transform and various technical lemmas.

\section{Key concepts and main result}
\label{sec:mainresult}
\subsection{Crystal setting and periodic trace} Let us consider a full-rank Bravais lattice $\Ell$ of $\R^d$, in the sense that there exists a linearly independent family $\bm a_1, \dots, \bm a_d \in \R^d$ such that
\[
\Ell = \left\{\sum_{j=1}^d \bm n_j \, \bm a_j , \; n \in \mathbb{Z}^d  \right\}.
\]
We consider a unit cell $\Gamma \subset \R^d$, such that $\{\Gamma + \ell\}_{\ell \in \Ell}$ forms a partition of $\R^d$, which interior is the one of the Wigner-Seitz cell associated to the lattice, defined by
\[
\{x \in \R^d \, | \; \forall \ell \in \Ell, \; |x| \leq  |x-\ell|\}.
\] 
Notice however the slight difference in the boundaries of $\Gamma$ and the Wigner–Seitz cell. For instance, in the case $\Ell = \mathbb{Z}^d$, the latter is $[-\frac12, \frac12]^d$ while one could take $\Gamma = [-\frac12, \frac12)^d$. As $\{\Gamma + \ell\}_{\ell \in \Ell}$ is a partition of $\R^d$, we can define, for any $z \in \R^d$,
\begin{equation} \label{eqdef:PGammaz}
    P_\Gamma z
\end{equation}
as the unique element of $\Gamma$ such that $z - P_\Gamma z \in \Ell$.
 In particular, one has
\begin{align} 
    \forall z \in \R^d \setminus (\partial \Gamma + \Ell), \qquad  \; \; \; \; \; P_\Gamma z \; &= z - \underset{\ell \in \Ell}{\mathrm{argmin}} |z-\ell|, \label{eqdef:Pgammainterior} \\
    \forall z \in \R^d \setminus (\partial \Gamma + \Ell), \qquad P_\Gamma (-z) &= -P_\Gamma z. \phantom{\forall z \in \R^d \setminus (\partial \Gamma + \Ell). \qquad}\label{eq:PgammaMinusz}
\end{align}
We denote
\begin{equation} \label{eqdef:gammPlusgammaMinus}
    \gamma_- = \inf_{z \in \partial \Gamma} |z|, \qquad \gamma_+ = \sup_{z \in \partial \Gamma} |z|.
\end{equation}
\begin{remark}
    In the case $\Ell = \mathbb{Z}^d$ and $\Gamma = [-\frac12,\frac12)^d$, we have $\gamma_- = \frac12$ and $\gamma_+ = \frac{\sqrt{d}}{2}$.
\end{remark}

\noindent We similarly choose a unit cell $\Gamma^*$ for the reciprocal lattice, the latter being defined as
\[
\Ell^* := \left\{\sum_{j=1}^d \bm n_j \, \bm b_j, \;\bm n:=(\bm n_j)_{1\leq j \leq d} \in \mathbb{Z}^d  \right\},
\]
where the vectors $\bm b_j\in \mathbb{R}^d$ for $1\leq j \leq d$ are such that
\[
\bm b_{j_1} \cdot \bm a_{j_2} = 2 \pi \, \indic_{j_1 = j_2}  \quad \forall 1\leq j_1, j_2 \leq d.
\]
In the following, we denote by $\mathcal{S}_\Ell$ the set of self-adjoint operators on $L^2(\R^d, \mathbb{C})$ that commute for all $\ell \in \Ell$ with the translation operator $\tau_\ell$ which is defined as follows: for any $\psi \in L^2(\R^d)$,
\[
\phantom{. y \in \R^d} \qquad \tau_\ell \psi(y) := \psi(y-\ell), \qquad y \in \R^d.
\]
Operators belonging to $\mathcal{S}_\Ell$ admit a ``Bloch decomposition'', detailed in Appendix~\ref{sec:Bloch}. We denote by $\ScEll$ the subset of $\mathcal{S}_\Ell$ of nonnegative operators. We may then define a notion of trace suited for ``periodic'' (in the sense that they commute with translations) operators of $\ScEll$ as follows.
\begin{definition}
  [\textbf{Periodic trace}] \label{def:periodicTrace} Let $R \in \ScEll$ and $\{R_k\}_{k \in \Gamma^*}$ be the family of operators on $L^2_{\rm per}(\Gamma)$ given by the Bloch transform of $R$ (see Appendix~\ref{sec:Bloch}). We define the \emph{periodic trace} of $R$ by
    \begin{equation} \label{eqdef:periodictrace}
        \perTr_{L^2(\R^d)}(R) := \fint_{\Gamma^*} \mathrm{Tr}_{L^2_{\rm per}(\Gamma)}(R_k) \, \dd k,
    \end{equation}
and let
\begin{equation} \label{eqdef:SL1}
\ScEllone := \{R \in \ScEll, \; \; \perTr_{L^2(\R^d)}(R) = 1\}.
\end{equation}
In order to lighten the notations in the following, $\perTr_{L^2(\R^d)}$ is denoted by $\perTr$ and $\mathrm{Tr}_{L^2_{\rm per}(\Gamma)}$ by $\mathrm{Tr}$, but one should keep in mind that these notions of trace are associated with operators acting on different spaces.
\end{definition}

\noindent For operators in $\ScEll$, the periodic trace is a more meaningful object to consider than the standard trace. Indeed, the latter may be viewed as counting the total number of electrons in $\R^d$, which is infinite in our setting, whereas the periodic trace measures the number of electrons per unit cell (see for instance the definition in~\cite{frank2020periodic}), which in our case is one.

\medskip

\noindent The classical counterpart to $\ScEll$ corresponds to the set of $\Ell$-periodized-in-$x$ probability densities on $\Gamma \times \R^d$, defined by
\begin{equation} \label{eqdef:PEll}
    \mathcal{P}_\Ell := \left\{f \in L^1_{\rm loc}(\R^d \times \R^d) \; \text{ s.t. } f \text{ is }\Ell\text{-periodic in }x \text{ and } \iint_{\Gamma \times \R^d} f(x,\xi) \, \dd x \, \dd \xi = 1 \right\}.
\end{equation} 

\subsection{Quantum-classical relationships: Töplitz operator and Husimi density}
As the core of this paper relies on comparing the quantum and classical dynamics, we need to be able to associate to a classical density in $\mathcal{P}_\Ell$ a quantum density operator in $\ScEllone$, which is the role of the Töplitz transformation, and conversely to a  quantum density operator in $\ScEllone$ a classical
 density in $\mathcal{P}_\Ell$, that is the role of the Husimi transform. Both rely on the notion of Schrödinger coherent state, and we present hereafter all these concepts in the $\Ell$-periodic setting. Periodized coherent states and Husimi density were already introduced in~\cite[Equations (15)--(16)]{batistic2019statistical}.

\begin{definition}
[\textbf{Schrödinger coherent state and $\Ell$-periodic Schrödinger coherent state}]  \label{def:periodizedCoherent} For any $(q,p) \in \R^d \times \R^d$, the Schrödinger coherent state $\coh{q}{p} \in L^2(\R^d)$ is
\begin{equation} \label{eqdef:cohState}
    \forall y \in \R^d, \quad  \coh{q}{p}(y)  := (\pi \hbar)^{-d / 4} \, e^{-\frac{|y-q|^2}{2 \hbar}} \, e^{i p \cdot y / \hbar}.      \quad \phantom{\forall y \in \R^d,}
\end{equation}
We define the $\Ell$-periodic Schrödinger coherent state $\percoh{q}{p} \in L^2_{\rm per}(\Gamma)$ by
\begin{equation} \label{eqdef:percohState}
    \forall y \in \R^d, \quad  \percoh{q}{p}(y)  := \sum_{\ell \in \Ell} \coh{q}{p}(y + \ell).      \quad \phantom{\forall y \in \R^d,}
\end{equation}
\end{definition}

\medskip

\noindent We show in Lemma~\ref{lem:schroProp} that $(q,p,y) \mapsto \percoh{q}{p}(y)$ is $\Ell$-periodic both in $q$ and $y$, $\|\percoh{q}{p}\|_{L^2_{\rm per}(\Gamma)} = 1$ and, for any $k \in \Gamma^*$,
\[
\percoh{q}{p - \hbar k} = \coh{q}{p}_k,
\]
where $(\coh{q}{p}_k)_{k\in \Gamma^*}$ is the Bloch decomposition of $\coh{q}{p} \in L^2(\R^d)$.
\begin{definition}  [\textbf{$\Ell$-periodic Töplitz operator}] \label{def:periodizedToplitz}
For each $f \in \mathcal{P}_\Ell$, we define the $\Ell$-Töplitz operator $T_\Ell [f]$ by, for any $\varphi \in L^2(\R^d)$,
\begin{equation} \label{eqdef:LperToplitz}
    T_\Ell [f] \varphi(y) := \iint_{\Gamma \times \R^d} \left(\fint_{\Gamma^*} e^{i k  \cdot y} \langle \underline{q,p - \hbar k} | \varphi_k \rangle_{L^2_{\rm per}(\Gamma)} \,  |\underline{q,p- \hbar k} \rangle(y)  \, \dd k \right) f(q,p)\, \dd q \, \dd p.
\end{equation}
\end{definition}

\medskip

\noindent It can be shown that $T_\Ell[f]$ belongs to $\ScEllone$ and that its Bloch decomposition writes, for any $k \in \Gamma^*$,
\[
T_\Ell [f]_k = \iint_{\Gamma \times \R^d}\proj {\underline{q,p - \hbar k}}\, f(q,p)\, \dd q \, \dd p.
\]

\begin{definition} [\textbf{$\Ell$-periodic Husimi density}] \label{def:periodizedHusimi}
Let $R \in \ScEllone$, and denote by $\{R_k\}_{k \in \Gamma^*}$ its Bloch decomposition. Define for any $(q,p) \in \R^d \times \R^d$
\begin{equation} \label{eqdef:husimi1}
\husimi[R](q,p) := \fint_{\Gamma^*} f_k(q,p) \, \dd k
\end{equation}
with, for any $k \in \Gamma^*$,
\begin{equation} \label{eqdef:husimi2}
 f_k(q,p) := (2 \pi \hbar)^{-d} \, 
 \langle\underline{q,p-\hbar k} | \, R_k \,  | \underline{q,p-\hbar k}\rangle_{L^2_{\rm per}(\Gamma)}.
\end{equation}
\end{definition}

\medskip

\noindent We show in Proposition~\ref{prop:HusimiIsAProba} that indeed, $\husimi[R] \in \mathcal{P}_\Ell$.

\subsection{Quantum and classical related dynamics} 
We consider a Lipschitz $\Ell$-periodic potential $V \in \mathcal{C}^{1}_{\rm Lip}(\R^d) \cap H^2_{\rm per}(\Gamma)$, from which we build a Hamiltonian operator on $L^2(\R^d)$,
\begin{equation} \label{eq:quantumHamiltonian}
\HH :=  -\frac{\hbar^2}{2} \Delta_y + V(y),
\end{equation}
which is self-adjoint and bounded from below with domain $H^2(\R^d)$, dense in $L^2(\R^d)$. Since $V$ is assumed to be $\Ell$-periodic and $\Delta_y$ commutes with any translation, we have $\HH \in \mathcal{S}_\Ell$. The corresponding von Neumann equation, for some ``periodic'' initial datum $R^{\rm in} \in \ScEllone$, writes
\begin{equation} \label{eq:vonNeumann}
    i \hbar \frac{\dd  R(t)}{\dd t} = \left[ \HH, \; R(t) \right], \quad \left.R\right|_{t=0} = R^{\rm in}, \quad R(t) \in \ScEllone.
\end{equation}
To the quantum Hamiltonian $\mathcal{H}$ can be associated its classical counterpart, the Hamiltonian function $H : \R^d \times \R^d \to \R$, defined for any $(x, \xi) \in \R^d \times \R^d$ by
\begin{equation} \label{eq:classicalHamiltonian}
    H(x,\xi) :=  \frac12|\xi|^2 + V(x).
\end{equation}
The classical dynamics associated to the Hamiltonian $H$ is the Liouville equation (counterpart to the von Neumann equation), which writes for some initial datum $f^{\rm in} \in \mathcal{P}_\Ell$,
\begin{equation} \label{eq:Liouville}
    \partial_t f(t,x,\xi) + \left\{H, \, f(t,\cdot,\cdot) \right\}(x,\xi) = 0, \qquad \left. f\right|_{t=0} = f^{\rm in},
\end{equation}
and $\{\bullet,\bullet\}$ stands for the Poisson bracket, with the convention
\begin{equation} \label{eqdef:Poissonbracket}
\left\{h, h_* \right\} = -\sum_{m = 1}^d \left( \frac{\partial h}{\partial x_m} \frac{\partial h_*}{\partial \xi_m} -   \frac{\partial h_*}{\partial x_m} \frac{\partial h}{\partial \xi_m}\right).
\end{equation}
The characteristics associated with the Liouville equation~\eqref{eq:Liouville} starting from $(x,\xi) \in \R^d \times \R^d$, denoted by $t\mapsto  (X(t;x,\xi),\Xi(t;x,\xi))$, satisfy the ODE
\begin{equation} \label{eq:trajectoryInit}
    \begin{cases}
    \dot{X} &= \phantom{-}\partial_{\xi}H(X,\Xi) = \Xi, \\
    \dot{\Xi} &= - \partial_{x}H(X,\Xi) = -\nabla V (X), \qquad \qquad (X(0), \Xi(0)) = (x,\xi).
    \end{cases}
\end{equation}

\subsection{Main result} Let us now state our main result. In Theorem~\ref{theo:obsfinal}, we obtain an observation inequality for the von Neumann equation in the crystal case, by combining the stability result, Theorem~\ref{theo:stability}, with the following à la Bardos-Lebeau-Rauch~\cite{bardos1992sharp} geometric condition~\eqref{eq:GC}, associated with the potential $V$, on a given triplet $(T,K,\Omega)$ on the classical dynamics, with $T>0$, $K \subset \Gamma \times \R^d$ compact and $\Omega \subset \Gamma$ open:
\begin{equation}
    \tag{GC} \label{eq:GC}
    \forall \, (x,\xi)\in K, \quad \exists \, t \in (0,T) \; \text{ s.t. } \; X(t;x,\xi) \in \Omega,
\end{equation}
where $X$ is the position component of the classical characteristics associated with the potential $V$, see~\eqref{eq:trajectoryInit}.

\medskip

\noindent Let us recall in the following lemma a crucial observability property for the Liouville equation.

\begin{lemma} [\textbf{Observability for Liouville}] \label{lem:liouvilleobs} Consider a triplet $(T,K,\Omega)$ satisfying the geometric condition~\eqref{eq:GC} for the potential $V$. There exists a constant $C_{\emph{\text{GC}}}[T,K,\Omega] > 0$ such that
\begin{equation} \label{eq:obsLiouville}
    \inf_{(x,\xi) \in K} \int_0^T \indic_{\Omega} ( X(t;x,\xi)) \, \text{\emph{d}} t  \geq C_{\emph{\text{GC}}}[T,K,\Omega].
\end{equation}
\end{lemma}
\noindent The proof of the above lemma can be found at the beginning of the proof of~\cite[Corollary 4.2]{golse2022quantitative}.

\begin{theorem} [\textbf{Observability for von Neumann}] \label{theo:obsfinal}
   Let $T>0$, $K \subset \Gamma \times \R^d$ a compact Borel set and $\Omega \subset \Gamma$ an open set satisfying the geometric condition~\eqref{eq:GC} associated with the potential $V \in \mathcal{C}^{1,1}(\R^d) \cap H^2_{\rm per}(\Gamma)$. Let $R^{\rm in} \in \ScEllone$ and $t \mapsto R(t)$ solving~\eqref{eq:vonNeumann} for the potential $V$. Then, for any $\delta > 0$, letting
\[
\Omega_\delta  = \Omega + B(0_{\R^d},\delta), \qquad \Omega_\delta^\Ell = \Omega_\delta + \Ell,
\]   
the solution to the von Neumann equation~\eqref{eq:vonNeumann} satisfies the following observability property on $[0,T] \times \Omega^\Ell_\delta$.

\medskip

\noindent \emph{Töplitz case.} If $R^{\rm in} = T_\Ell [f^{\rm in}]$ with $f^{\rm in} \in \mathcal{P}_\Ell$, the $\Ell$-periodic Töplitz operator defined in Definition~\ref{def:periodizedToplitz},
\begin{multline}
\label{eq:theoremToplitz_Obs}
    \int_0^T \underline{\emph{\text{Tr}}} \left(\indic_{\Omega_\delta^\Ell}  \,  R(t) \, \indic_{\Omega_\delta^\Ell}\right) \, \emph{\text{d}} t \geq C_{\emph{\text{GC}}}[T,K,\Omega] \iint_{(x,\xi) \in K}   f^{\rm in}(x,\xi) \, \dd x \, \dd \xi \\
    - C_{\text{\emph{Töplitz}}}[\Gamma,T,\mathrm{Lip}(\nabla V)] \; \frac{\sqrt{d\hbar}}{\delta},
\end{multline}
with, recalling that $\gamma_-$ and $\gamma_+$ are defined in~\eqref{eqdef:gammPlusgammaMinus},
\begin{equation} \label{eqdef:CToplitz}
C_{\text{\emph{Töplitz}}}[\Gamma,T,L] := \sqrt{\frac{\gamma_-}{2\gamma_+}} \; \inf_{\lambda > 0} \frac{\exp \left(\frac{2\gamma_+}{\gamma_-} \left(\lambda + \frac{L^2}{\lambda} \right) T \right)-1}{\lambda^2 + L^2} \, \sqrt{\frac{1 + \lambda^2}{2}}.
\end{equation}

\noindent \emph{Pure state case.} If the Bloch decomposition $\{R^{\rm in}_k\}_{k \in \Gamma^*}$ of $R^{\rm in}$ satisfies $R^{\rm in}_k = |u^{\rm in}_k \rangle \langle u^{\rm in}_k|$ for any $k \in \Gamma^*$ and
\begin{equation} \label{eqdef:cBold}
\mathbf{c}_{u^{\rm in}} := \fint_{\Gamma^*}\|u^{\rm in}_k\|^4_{L^2_{\rm per}(\Gamma)} \, \dd k < +\infty,
\end{equation}
then
\begin{multline}
\label{eq:theoremPure_Obs}
    \int_0^T \underline{\emph{\text{Tr}}} \left(\indic_{\Omega_\delta^\Ell} \,  R(t) \, \indic_{\Omega_\delta^\Ell}\right) \, \emph{\text{d}} t \geq C_{\emph{\text{GC}}}[T,K,\Omega] \iint_{(x,\xi) \in K}   \husimi[R^{\rm in}](x,\xi) \, \dd x \, \dd \xi \\
    - C_{\text{\emph{pure}}}[\Gamma,T,\mathrm{Lip}(\nabla V)] \; \frac{\sqrt{d\hbar \,\mathbf{c}_{u^{\rm in}} + 2{\Delta_{\Gamma, \hbar}}^2 (R^{\rm in})}}{\delta},
\end{multline}
where $\husimi[R^{\rm in}]$ is the $\Ell$-periodic Husimi transform of $R$ defined in Definition~\ref{def:periodizedHusimi},  $\Delta_{\Gamma, \hbar} (R^{\rm in})$ is the standard deviation of $R^{\rm in}$ defined by 
\begin{equation} \label{eqdef:stdDevR}
    \Delta_{\Gamma, \hbar} (R^{\rm in}) := \left(\frac12 \fint_{\Gamma^*} \iint_{\Gamma \times \Gamma} \left\{ |P_\Gamma(y-q)|^2 + \left(i\hbar \nabla_y - i \hbar \nabla_q \right)^2\right\} \, |u_k^{\rm in}(y)|^2 \, |u_k^{\rm in}(q)|^2 \,  \dd y \, \dd q \,\dd k \right)^{\frac12},
\end{equation}
and $C_\text{\emph{pure}}$ writes 
\begin{equation} \label{eqdef:Cpure}
    C_\text{\emph{pure}}[\Gamma, T,L] :=\sqrt{\frac{\gamma_-}{2\gamma_+}} \; \frac{\exp  \left(\frac{2\gamma_+}{\gamma_-} (1 + L^2) \, T \right) - 1}{1 + L^2}.
\end{equation}
\end{theorem}

\medskip

\noindent We can derive an observability inequality from Theorem~\ref{theo:obsfinal}. The right-hand side of Equation~\eqref{eq:theoremToplitz_Obs} can be lower-bounded by a positive constant for any Töplitz initial datum which symbol has support in $K$, as soon as
\begin{equation}
\label{eq:constantToplitz_Obs}
\hbar < \frac{\delta^2}{d}  \times  \frac{C_{\text{GC}}[T,K,\Omega]^2}{C_{\text{Töplitz}}[\Gamma,T,\mathrm{Lip}(\nabla V)]^2}.
\end{equation}
On the other hand, there is a non-empty set of ``periodic pure state'' initial data $R^{\rm in}$, uniform in $\hbar \in (0, \hbar_0)$ for some $\hbar_0 >0$, such that~\eqref{eq:theoremPure_Obs} is lower-bounded by a positive constant: indeed, if $R^{\rm in}$ is constructed from the periodic Schrödinger coherent states $\{\percoh{q}{p - \hbar k}\}_{k \in \Gamma^*}$ defined in~\eqref{eqdef:percohState} with $(q,p) \in \overset{\circ}{K}$, the term $d\hbar \,\mathbf{c}_{u^{\rm in}} + 2{\Delta_{\Gamma, \hbar}}^2 (R^{\rm in})$ is of the order of $\hbar$, while $\iint_{K}\husimi[R^{\rm in}](x,\xi) \, \dd x \, \dd \xi$ may be lower bounded by a positive constant depending only on $\hbar_0$.
\begin{remark}
    The observation inequalities~\eqref{eq:theoremToplitz_Obs} and~\eqref{eq:theoremPure_Obs} involve the set $\Omega^\Ell_\delta = \Omega_\delta + \Ell$, however this does not represent a simultaneous observation of the system in all copies of  $\Omega_\delta$, but rather a single observation of the system in $\Omega_\delta$ while knowing its periodicity. To illustrate this fact, consider a ``periodic pure state'' $R$, in the sense that its Bloch decomposition can be written
    \[
    R_k = |v_k \rangle \langle v_k|, \qquad k \in \Gamma^*,
    \]
    with $v_k \in L^2_{\rm per}(\Gamma)$. The quantity representing the observation of the system writes
\begin{multline*}
\underline{\emph{\text{Tr}}}_{L^2(\R^d)} \left(\indic_{\Omega_\delta^\Ell} \,  R \, \indic_{\Omega_\delta^\Ell}\right) = \fint_{\Gamma^*} \emph{\text{Tr}}_{L^2_{\emph{\text{per}}}(\Gamma)} \left(\indic_{\Omega_\delta^\Ell} \,  R_k \, \indic_{\Omega_\delta^\Ell}\right) \, \emph{\text{d}} k = \fint_{\Gamma^*} \int_{\Gamma} \indic_{\Omega_\delta}(x) \,|v_k(x)|^2 \, \emph{\text{d}} x \, \emph{\text{d}} k \\
= \int_{\Omega_\delta} \, \rho_R(x) \, \emph{\text{d}}x,
\end{multline*}
where $\displaystyle x \mapsto \rho_R(x) := \fint_{\Gamma^*} |v_k(x)|^2 \, \dd k$ is interpreted as the ($\Ell$-periodic) particle density. The above expresses the idea that even if the observation inequality involves $\Omega_\delta^\Ell$, the actual observation is indeed performed only once in $\Omega_\delta$.
\end{remark}

\medskip

\subsection{Structure of the proof} To prove Theorem~\ref{theo:obsfinal}, we first show a quantum/classical stability property between solutions to the von Neumann equation~\eqref{eq:vonNeumann} and to the Liouville equation~\eqref{eq:Liouville} in Theorem~\ref{theo:stability}. To quantify of form of ``distance'' between the quantum and classical solutions, we rely on an optimal-transport-like pseudo-metric, defined in Subsection~\ref{subsec:optimaltransp}. Once stability is obtained, yielding a control of the pseudo-metric at time $t$ by its value at time $0$, we consider suitable quantum/classical pairs at initial time, using notions of Töplitz and Husimi transforms (Definitions~\ref{def:periodizedToplitz} and~\ref{def:periodizedHusimi}) allowing to bound the discrepancy between the quantum and the classical observability quantities (Propositions~\ref{prop:boundToplitz} and~\ref{prop:EhbarHusimi}). We conclude to Theorem~\ref{theo:obsfinal} using the observability property for the Liouville equation (Lemma~\ref{lem:liouvilleobs}).

\section{Liouville-von Neumann stability} \label{sec:stability}

In this section, we relate the quantum von Neumann equation~\eqref{eq:vonNeumann} with its classical counterpart, the Liouville equation. We introduce the notion of classical/quantum optimal transport pseudo-metric in the periodic setting, based on the one of \cite{golse2022quantitative}. This allows us to obtain a notion of Liouville-von Neumann stability in Theorem~\ref{theo:stability}, a periodic counterpart to \cite[Theorem 3.1]{golse2022quantitative}.

\subsection{Optimal-transport-like pseudo-metric} \label{subsec:optimaltransp} We first introduce the notion of coupling between a classical density $f \in \mathcal{P}_\Ell$ and a quantum density operator $R \in \ScEll$.
\begin{definition} [\textbf{Classical-Quantum coupling}] \label{def:class-quandcoupl}
    A coupling between $f \in \mathcal{P}_\Ell$ and $R \in \ScEll$ is a measurable operator-valued function
    \begin{equation*}
        Q : \R^d \times \R^d  \longrightarrow  \ScEll
    \end{equation*}
    which is $\Ell$-periodic in its first variable $x$ and satisfies the ``marginal'' constraints
    \begin{equation} \label{eq:marginalconstraintscoupling}
        \forall \, (x,\xi) \in \Gamma \times \R^d, \quad \perTr(Q(x,\xi)) = f(x,\xi) \qquad \text{and } \qquad \iint_{\Gamma \times \R^d} Q(x, \xi) \, \dd x \, \dd \xi = R.
    \end{equation}
    The set of couplings between $f$ and $R$ is denoted by $\C(f,R)$.
\end{definition}

\medskip

\noindent We now introduce the notion of quadratic transportation cost, as an analogue to the the one introduced in~\cite{golse2022quantitative}, in the $\Ell$-periodic setting. We highlight that the $\Ell$-periodicity is ensured by the consideration of $|P_\Gamma(x-y)|^2$ instead of $|x-y|^2$. However, $y \mapsto |P_\Gamma(x-y)|^2 \notin H^2_{\rm per}(\Gamma)$ and such regularity is necessary in the proof. Therefore, a regularization function $\Theta$ is introduced, ensuring $y \mapsto \Theta(|P_\Gamma(x-y)|^2) \in H^2_{\rm per}(\Gamma)$.

\medskip

\begin{definition}
    [\textbf{Transportation cost}] \label{def:class-quandcost} For any $\lambda > 0$, we define the transportation costs $c_\lambda$ and $\hat{c}_\lambda$ as operator-valued functions of $(x,\xi) \in \R^d \times \R^d$, such that $c_\lambda(x,\xi)$ acts on $L^2(\R^d)$ and $\hat{c}_\lambda(x,\xi)$ acts on $L^2_{\rm per}(\Gamma)$ as
\begin{equation} \label{eqdef:transpcost}
\begin{split}
    \forall \psi \in L^2(\R^d), \qquad    c_{\lambda}(x,\xi) \, \psi \,(y) := \left\{\lambda^2 \, \Theta(|P_{\Gamma}(x-y)|^2) + (\xi + i \hbar \nabla_y)^2 \right\} \, \psi \,(y), \\
      \forall \varphi \in L^2_{\rm per}(\Gamma), \qquad    \hat{c}_{\lambda}(x,\xi) \, \varphi \, (y) := \left\{\lambda^2 \, \Theta(|P_{\Gamma}(x-y)|^2) + (\xi + i \hbar \nabla_y)^2 \right\}  \, \varphi \, (y).
\end{split}
\end{equation}
In the above formula, $(\xi + i \hbar \nabla_y)^2$ stands for $(\xi + i \hbar \nabla_y) \cdot (\xi + i \hbar \nabla_y)$ and, recalling the notation $\gamma_- = \inf_{z \in \partial \Gamma} |z|$, for any $r \geq 0$,
\begin{equation} \label{eqdef:Theta}
    \Theta (r) := \int_0^r \left( 1 - \frac{r_*}{\gamma_{-}}\right)_+ \dd r_*.
\end{equation}
Both $c_\lambda$ and $\hat{c}_\lambda$ are $\Ell$-periodic in $x$ and $c_\lambda \in \ScEll$.
\end{definition} 
\noindent Note that $y\mapsto \Theta(|P_{\Gamma}(x-y)|^2) \in H^2_{\rm per}(\Gamma)$ and, recalling that $\gamma_+ = \sup_{z \in \partial \Gamma} |z|$, we have for any $x,y \in \R^d$,
\begin{equation}
\label{eq:equivalenceThetaPGammasqr}
    \frac{\gamma_-}{2\gamma_+} \, |P_{\Gamma}(x-y)|^2 \leq \Theta(|P_{\Gamma}(x-y)|^2) \leq |P_{\Gamma}(x-y)|^2 \quad \text{and} \quad 0 \leq \Theta' \leq 1.
\end{equation}
With the notions of periodic trace (Definition~\ref{def:periodicTrace}), couplings and transportation cost at hand, we define the transportation classical/quantum pseudo-metric:

\begin{definition} [\textbf{Transportation pseudo-metric}] \label{def:class-quandpseudometric} For each $f \in \mathcal{P}_\Ell$ and $R \in \ScEllone$, we set
\begin{equation} \label{eqdef:pseudometric}
    E_{\hbar,\lambda}(f,R) := \inf_{Q \in \C(f,R)} \left(\iint_{\Gamma \times \R^d} \perTr \left(Q(x,\xi)^\frac12 \, c_{\lambda}(x,\xi) \, Q(x,\xi)^\frac12 \right) \dd x \, \dd \xi \right)^\frac12.
\end{equation}
\end{definition}

\noindent It is well-defined (could be infinite) as for any $x,\xi$, we indeed have $Q(x,\xi)^\frac12 \, c_{\lambda}(x,\xi) \, Q(x,\xi)^\frac12 \in \ScEll$.

\subsection{Flows and Bloch flows} \label{subsec:flowsBlochflows} Working in the $\Ell$-periodic crystal setting requires to make use of the Bloch decomposition. In this subsection, we present the Bloch decomposition $(\HH_k)_{k \in \Gamma^*}$ of the Hamiltonian operator $\HH$, the corresponding family of classical Hamiltonian functions $(H_k)_{k \in \Gamma^*}$, as well as the family of quantum flow operators $(U_k(t))_k$ and classical flow applications $(\Phi_{k,t})_k$ associated to the corresponding, respectively, quantum and classical dynamics. Recall that the quantum Hamiltonian associated with our problem is
\begin{equation*} 
\HH :=  -\frac{\hbar^2}{2} \Delta_y + V(y),
\end{equation*}
which, as an operator of $\ScEll$, admits a Bloch decomposition which writes (by direct computation, or see~\cite[Chapter 7]{lewin2024spectral}), for any $k \in \Gamma^*$,
\begin{equation} \label{eq:quantumHamiltonian_k}
\HH_k :=  \frac{1}{2}(-i \hbar\nabla_y + \hbar k)^2 + V(y).
\end{equation}
The associated von Neumann flow operators are denoted respectively $t \mapsto U(t)$ and $t \mapsto U_k(t)$. Remark that, while $U(t)$ acts on $L^2(\R^d)$, $U_k(t)$ acts on $L^2_{\rm per}(\Gamma)$. The classical counterparts to~\eqref{eq:quantumHamiltonian} and~\eqref{eq:quantumHamiltonian_k} respectively write, for any $(x, \xi) \in \R^d \times \R^d$,
\begin{equation*}
    H(x,\xi) :=  \frac12|\xi|^2 + V(x),
\end{equation*}
and
\begin{equation}
    H_k(x,\xi) :=  \frac12|\xi + \hbar k|^2 + V(x).
\end{equation}
We define, for any $(x,\xi) \in \R^d \times \R^d$, $t \mapsto \Phi_t(x,\xi)$ as the flow associated with the ODE
\begin{equation} 
    \begin{cases}
    \dot{X} &= \phantom{-}\partial_{\xi}H(X,\Xi) = \Xi, \\
    \dot{\Xi} &= - \partial_{x}H(X,\Xi) = -\nabla V (X), \qquad \qquad (X(0), \Xi(0)) = (x,\xi).
    \end{cases}
\end{equation}
Then, for any $k \in \Gamma^*$, the $k$-flow $\Phi_{k,\cdot}(x,\xi)$ defined by

\begin{equation} \label{eqdef:flow_k}
 \forall t \in \R, \qquad    \Phi_{k,t}(x,\xi) := \Phi_t(x,\xi + \hbar k) - (0, \, \hbar k) \phantom{ \qquad \forall t \in \R, }
\end{equation}

\medskip

\noindent is the flow associated to the system, indexed by $k \in \Gamma^*$,
    \begin{equation} \label{eq:trajectory_k}
        \begin{cases}
        \dot{X}_k &= \phantom{-}\partial_{\xi}H_k(X_k,\Xi_k) = \Xi_k + \hbar k,\\
        \dot{\Xi}_k &= -\partial_{x}H_k(X_k,\Xi_k) = -\nabla V(X_k),
        \end{cases} \qquad \qquad (X_k(0), \, \Xi_k(0)) = (x,\xi).
    \end{equation}
Note that it effectively appears in~\eqref{eq:trajectory_k} that the effective momentum $\dot{X}_k$ is the standard momentum $\Xi_k$ plus the crystal momentum $\hbar k$.

\medskip

\noindent Note that the $\Ell$-periodicity of $V$ implies that for any $(t,x,\xi) \in \R \times \R^d \times \R^d$ and $\ell \in \Ell$, we have
\begin{equation} \label{eq:pseudoPeriodicityPhi}
    \Phi_t(x+\ell,\xi) = \Phi_t(x,\xi) + (\ell,0).
\end{equation}

\subsection{Stability result} 
 We now state our Liouville-von Neumann stability result.

\medskip

\begin{theorem} [\textbf{Liouville-von Neumann stability}] \label{theo:stability}
    Let $f^{\rm in} \in \mathcal{P}_\Ell$ and $R^{\rm in} \in \ScEllone$. For all $t \geq 0$, set
    \begin{equation*}
        R(t) = U(t)^* \, R^{\rm in} \, U(t),  \qquad f(t, X, \Xi):=f^{\rm in} \left( \Phi_{-t}(X, \Xi) \right) \quad \text{for a.e. } (X,\Xi) \in \R^d \times \R^d.
    \end{equation*}
Then for each $\lambda > 0$ and $t \geq 0$, one has
\begin{equation}
    E_{\hbar,\lambda}(f(t),R(t)) \leq E_{\hbar,\lambda}(f^{\rm in},R^{\rm in}) \; \exp \left(\eta_{\lambda,V} \, t \right),
\end{equation}
with $\eta_{\lambda,V} := \frac{\gamma_+}{\gamma_-} \left(\lambda + \frac{\mathrm{Lip}(\nabla V)^2}{\lambda} \right)$.
\end{theorem}
\begin{proof}
    Let $Q^{\rm in} \in \C(f^{\rm in}, R^{\rm in})$. Set, for any $t \in \R$ and $(X,\Xi) \in \R^d \times \R^d$,
    \[
    Q(t,X,\Xi) := U(t)^* \, Q^{\rm in} \circ \Phi_{-t} (X,\Xi) \, U(t)
    \]
 and
\begin{equation} \label{eqdef:fullEnergyproof}
\E(t) := \iint_{\Gamma \times \R^d} \perTr \left(Q(t,X,\Xi)^{\frac12} \, c_\lambda (X,\Xi) \, Q(t,X,\Xi)^\frac12 \right) \, \dd X \, \dd \Xi.
\end{equation}
    Note that $Q^{\rm in}$ $\Ell$-periodic in $x$ and~\eqref{eq:pseudoPeriodicityPhi} imply that $X \mapsto Q^{\rm in} \circ \Phi_{-t}(X,\Xi)$ is also $\Ell$-periodic in $X$. Hence, we indeed have $Q(t) \in \C(f(t),R(t))$, and in particular,
    \[
    \E(t) \geq E_{\hbar,\lambda}(f(t),R(t)), \qquad t \geq 0.
    \]
    Since $c_\lambda$ and $Q(t,\cdot,\cdot)$ are $\Ell$-periodic in the $x$-variable, so is
    \[
    (X,\Xi) \mapsto \perTr \left(Q(t,X,\Xi)^{\frac12} \, c_\lambda (X,\Xi) \, Q(t,X,\Xi)^\frac12 \right).
    \]
    Applying then Lemma~\ref{lem:LiouvilleChangeVar} (changing variables $(X, \Xi) \mapsto (x,\xi) = \Phi_{-t}(X,\Xi)$) yields
    \[
    \E(t) := \iint_{\Gamma \times \R^d} \perTr \left(Q^{\rm in}(x,\xi)^{\frac12} \, U(t) \, c_\lambda (\Phi_t(x,\xi)) \, U(t)^* \,Q^{\rm in}(x,\xi)^{\frac12} \right) \, \dd x \, \dd \xi.
    \] 
    Let us now consider the family of operators $\{Q^{\rm in}_k(x,\xi)\}_{k \in \Gamma^*}$, the Bloch transform of $Q^{\rm in}(x,\xi)$ (which we recall belongs to $\ScEllone$).  By Corollary~\ref{cor:fundamentalPropertyBlochOp}, we have that for any $k \in \Gamma^*$,
\begin{multline}
    \left(Q^{\rm in}(x,\xi)^{\frac12} \, U(t) \, c_\lambda (\Phi_t(x,\xi)) \, U(t)^* \,Q^{\rm in}(x,\xi)^{\frac12}\right)_k \\
    = Q^{\rm in}_k(x,\xi)^{\frac12} \, (U(t))_k \, \left( c_\lambda (\Phi_t(x,\xi)) \right)_k \, \left(U(t)^* \right)_k \, Q^{\rm in}_k(x,\xi)^{\frac12},
\end{multline}
which, using Lemma~\ref{lem:BlochTransFormulas} and the fact that $\Phi_t(x,\xi) - (0,\hbar k) = \Phi_{k,t}(x,\xi - \hbar k)$, is exactly
\[
Q^{\rm in}_k(x,\xi)^{\frac12} \, U_k(t) \,  \hat{c}_\lambda (\Phi_{k,t}(x,\xi - \hbar k)) \, U_k(t)^* \, Q^{\rm in}_k(x,\xi)^{\frac12},
\]
so that
    \begin{equation} \label{eq:trace_k}
    \E(t) = \iint_{\Gamma \times \R^d} \fint_{\Gamma^*} \mathrm{Tr} \left(Q^{\rm in}_k(x,\xi)^{\frac12} \, U_k(t) \, \hat{c}_\lambda (\Phi_{k,t}(x,\xi - \hbar k)) \, U_k(t)^* \,Q^{\rm in}_k(x,\xi)^{\frac12} \right) \, \dd k \, \dd x \, \dd \xi.
    \end{equation}
Exchanging integrals in~\eqref{eq:trace_k} we get
\[
\E(t) =  \fint_{\Gamma^*} \E_k(t) \, \dd k,
\]
with
\[
\E_k(t) := \iint_{\Gamma \times \R^d} \mathrm{Tr} \left(Q^{\rm in}_k(x,\xi)^{\frac12} \, U_k(t) \, \hat{c}_\lambda (\Phi_{k,t}(x,\xi - \hbar k) ) \, U_k(t)^* \,Q^{\rm in}_k(x,\xi)^{\frac12} \right) \, \dd x \, \dd \xi.
\]
We obtain, performing the change of variable $\xi \mapsto \xi - \hbar k$ still denoted by $\xi$ for convenience,
\[
\E_k(t) := \iint_{\Gamma \times \R^d} \mathrm{Tr} \left(Q^{\rm in}_k(x,\xi + \hbar k)^{\frac12} \, U_k(t) \, \hat{c}_\lambda (\Phi_{k,t}(x,\xi)) \, U_k(t)^* \, Q^{\rm in}_k(x,\xi + \hbar k)^{\frac12} \right) \, \dd x \, \dd \xi.
\]
For each $k \in \Gamma^*$ and $(x,\xi) \in \R^d \times \R^d$, we consider $(e_j(x,\xi,k))_{j \in \N}$ the $L^2_{\rm per}(\Gamma)$-complete orthonormal system of eigenvectors of $Q^{\rm in}_k(x,\xi + \hbar k)$, which we know from Proposition~\ref{prop:blochTfunbd} is a Trace-class nonnegative self-adjoint operator on $L^2_{\rm per}(\Gamma)$, associated with the eigenvalues $(\rho_j(x,\xi,k))_{j \in \N}$. Straightforwardly, both $e_j$ and $\rho_j$ are $\Ell$-periodic in the variable $x$. It then comes that
\begin{multline}
\mathrm{Tr} \left(Q^{\rm in}_k(x,\xi + \hbar k)^{\frac12} \, U_k(t) \, \hat{c}_\lambda (\Phi_{k,t}(x,\xi)) \, U_k(t)^* \, Q^{\rm in}_k(x,\xi + \hbar k)^{\frac12} \right) =\\
\sum_{j \in \N} \rho_j (x ,\xi, k) \, \left\langle U_k(t) e_j(x, \, \xi, k) |\hat{c}_\lambda(\Phi_{k,t}(x,\xi)) | U_k(t) e_j(x, \, \xi, k)  \right\rangle.
\end{multline}
We highlight that the brackets are taken in $L^2_{\rm per}(\Gamma)$. For any $\phi \in \mathcal{C}^\infty_{\rm per}(\Gamma)$, the map
\[
t \mapsto \left\langle U_k(t) \phi |\hat{c}_\lambda(\Phi_{k,t}(x,\xi)) | U_k(t) \phi \right\rangle
\]
is $\mathcal{C}^1$ on $\R$ and, by the same manipulations as in~\cite[proof of Theorem 3.1]{golse2022quantitative}, one has, recalling that $t \mapsto U_k(t)$ is the quantum flow associated to $\HH_k$ and $t \mapsto \Phi_{k,t}$ is the  classical flow associated with $H_k$,
\begin{multline*}
    \frac{\dd }{\dd t} \left\langle U_k(t) \phi |\hat{c}_\lambda(\Phi_{k,t}(x,\xi)) | U_k(t) \phi \right\rangle = \left\langle U_k(t) \phi \left| \frac{i}{\hbar} \left[\mathcal{H}_k, \, \hat{c}_\lambda(\Phi_{k,t}(x,\xi) )\right] \right| U_k(t) \phi \right\rangle \\
    + \left\langle U_k(t) \phi |\left\{H_k, \, \hat{c}_\lambda \right\} (\Phi_{k,t}(x,\xi)) | U_k(t) \phi \right\rangle.
\end{multline*}
Let us point out that for all $x,y \in \R^d$, we have
\[
\nabla_x \;  \Theta(|P_{\Gamma}(x-y)|^2) = 2 P_{\Gamma}(x-y) \Theta'(|P_{\Gamma}(x-y)|^2).
\]
On the one hand, denoting $(X_k, \Xi_k) = \Phi_{k,t}(x,\xi)$, we have
\begin{align*}
\left\{H_k, \, \hat{c}_\lambda \right\} (X_k, \Xi_k) &= - \nabla_x H_k (X_k,\Xi_k) \cdot  \nabla_\xi \hat{c}_\lambda (X_k,\Xi_k) + \nabla_\xi H_k (X_k,\Xi_k) \cdot  \nabla_x \hat{c}_\lambda (X_k,\Xi_k)\\
&= -\nabla V(X_k) \cdot 2 \left(\Xi_{k} + i \hbar \nabla_y \right) - \left(\Xi_{k} + \hbar k \right) \cdot \left[2 \, \lambda^2  \, \Theta'(|P_{\Gamma}(X_k-y)|^2) \, P_{\Gamma}(X_k-y) \right].
\end{align*}
On the other hand,
\begin{align*}
    \frac{i}{\hbar} \left[\mathcal{H}_k, \, \hat{c}_\lambda(X_k, \Xi_k)  \right] &= \frac{i}{\hbar} \left[V(y), \, (\Xi_k + i \hbar \nabla_y)^2 \right] \\
    &+ \frac{i}{\hbar} \left[V(y), \,\lambda^2 \Theta(|P_{\Gamma}(X_k-y)|^2) \right] \\
&+\frac{i}{\hbar} \left[\frac{1}{2}(-i \hbar\nabla_y + \hbar k)^2, \, (\Xi_k + i \hbar \nabla_y)^2 \right] \\
&+\frac{i}{\hbar} \left[\frac{1}{2}(-i \hbar\nabla_y + \hbar k)^2, \, \lambda^2  \Theta(|P_{\Gamma}(X_k-y)|^2) \right],
\end{align*}
which, thanks to Lemma~\ref{lem:commutators} in Appendix~\ref{app:lemmas}, yields
\begin{align*}
    \frac{i}{\hbar} \left[\mathcal{H}_k, \, \hat{c}_\lambda(X_k, \Xi_k )  \right] &= ( \Xi_k + i \hbar \nabla_y) \cdot \nabla V(y) + \nabla V(y) \cdot \left( \Xi_k + i \hbar \nabla_y \right) \\
&+\lambda^2 (-i\hbar \nabla_y + \hbar k) \cdot \left[ P_{\Gamma}(y-X_k) \; \Theta'(|P_\Gamma(X_k-y)|^2 ) \right] \\
&+ \lambda^2 P_{\Gamma}(y-X_k) \; \Theta'(|P_\Gamma(X_k-y)|^2 ) \cdot (-i\hbar \nabla_y + \hbar k)
\end{align*}
All in all, 
\begin{align*}
   \big\{H_k, \, \hat{c}_\lambda \big\} (X_k, \Xi_k) + \frac{i}{\hbar} &\left[\mathcal{H}_k, \, \hat{c}_\lambda(X_k, \Xi_k )  \right] = \\
   &- \left(\Xi_{k} + i \hbar \nabla_y \right)  \cdot \nabla V(X_k) -  \nabla V(X_k)  \cdot    \left(\Xi_{k} + i \hbar \nabla_{y} \right) \\
   &+( \Xi_k + i \hbar \nabla_y) \cdot \nabla V(y) + \nabla V(y) \cdot \left( \Xi_k + i \hbar \nabla_y \right) \\
   &+\lambda^2 \left(\Xi_{k} + \hbar k \right) \cdot \left[ P_{\Gamma}(X_k-y) \, \Theta'(|P_\Gamma(X_k-y)|^2 ) \right] \\
 &+ \lambda^2 \left[P_{\Gamma}(X_k-y) \; \Theta'(|P_\Gamma(X_k-y)|^2 ) \right] \cdot \left(\Xi_{k} + \hbar k \right) \\
&+\lambda^2 (-i\hbar \nabla_y + \hbar k) \cdot \left[ P_{\Gamma}(y-X_k) \, \Theta'(|P_\Gamma(y-X_k)|^2 ) \right] \\
&+ \lambda^2 \left[P_{\Gamma}(y-X_k) \; \Theta'(|P_\Gamma(y-X_k)|^2 )  \right] \cdot (-i\hbar \nabla_y + \hbar k).
\end{align*}
Using the identity $P_{\Gamma}(-z) = -P_{\Gamma}z$ almost everywhere (in $\R^d \setminus (\partial \Gamma + \Ell)$), coming from $-\Ell = \Ell$, the above equals
\begin{align*}
&-\left(\Xi_{k} + i \hbar \nabla_{y} \right) \cdot (\nabla V(X_k) - \nabla V(y)) +  (\nabla V(X_k) - \nabla V(y)) \cdot  \left(\Xi_{k} + i \hbar \nabla_{y} \right) \\
&+\lambda^2  \left(\Xi_{k} + i \hbar \nabla_y \right) \cdot \left[ P_{\Gamma}(X_k-y) \; \Theta'(|P_{\Gamma}(X_k-y)|^2) \right] \\
&+ \lambda^2 \, P_{\Gamma}(X_k-y)  \, \Theta'(|P_{\Gamma}(X_k-y)|^2) \cdot \left(\Xi_{k} + i \hbar \nabla{y} \right).
\end{align*}
Using Peter-Paul's inequality, by which for positives $x,y$ it holds that $2 xy \leq \lambda x^2 + \lambda^{-1}y^2$, the above is less than
\begin{multline*}
\lambda \left(\lambda^2 \, \left|\Theta'(|P_{\Gamma}(X_k-y)|^2) \, P_{\Gamma}(X_k-y)  \right|^2 + (\Xi_{k} + i \hbar \nabla_{y})^2 \right) \\
+ \frac{1}{\lambda}  \left(\lambda^2 |\nabla V(X_k) - \nabla V(y)|^2 + (\Xi_{k} + i \hbar \nabla_{y})^2 \right).
\end{multline*}
Since $0 \leq \Theta' \leq 1$ and $V$ is $\Ell$-periodic, the above is itself less than
\begin{equation*}
\lambda  \left(\lambda^2 |P_{\Gamma}(X_k-y) |^2 + (\Xi_{k} + i \hbar \nabla_{y})^2 \right) + \frac{\mathrm{Lip}(\nabla V)^2}{\lambda} \left(\lambda^2 |P_{\Gamma}(X_k-y) |^2 + (\Xi_{k} + i \hbar \nabla_{y})^2 \right), 
\end{equation*}
which, using the fact that $\displaystyle |P_{\Gamma}(X_k-y) |^2 \leq \frac{2 \gamma_+}{\gamma_-} \, \Theta(|P_{\Gamma}(X_k-y) |^2)$, is less than 
\[
\frac{2 \gamma_+}{\gamma_-}\left( \lambda + \frac{\mathrm{Lip}(\nabla V)^2}{\lambda} \right) \, \hat{c}_\lambda(X_k,\Xi_k).
\]
As $(X_k,\Xi_k)$ is by definition $\Phi_{k,t}(x,\xi)$, we have thus proven that, for any $t \geq 0$,
\[
\frac{\dd }{\dd t} \left\langle U_k(t) \phi |\hat{c}_\lambda(\Phi_{k,t}(x,\xi)) | U_k(t) \phi \right\rangle \leq \frac{2 \gamma_+}{\gamma_-}\left( \lambda + \frac{\mathrm{Lip}(\nabla V)^2}{\lambda} \right) \left\langle U_k(t) \phi |\hat{c}_\lambda(\Phi_{k,t}(x,\xi))| U_k(t) \phi \right\rangle.
\]
Similar density, Grönwall and linearity arguments as in~\cite[proof of Theorem 3.1]{golse2022quantitative} then lead to
\begin{multline*}
\mathrm{Tr} \left(Q^{\rm in}_k(x,\xi + \hbar k)^{\frac12} \, U_k(t) \, \hat{c}_\lambda (\Phi_{k,t}(x,\xi)) \, U_k(t)^* \, Q^{\rm in}_k(x,\xi + \hbar k)^{\frac12} \right) \\
\leq \exp \left(\eta_{\lambda,V} \, t \right) \, \mathrm{Tr} \left(Q^{\rm in}_k(x,\xi + \hbar k)^{\frac12} \, \hat{c}_\lambda (x,\xi) \, Q^{\rm in}_k(x,\xi + \hbar k)^{\frac12} \right),
\end{multline*}
  with $\eta_{\lambda,V} := \frac{2 \gamma_+}{\gamma_-} \left(\lambda + \frac{\mathrm{Lip}(\nabla V)^2}{\lambda} \right)$, implying, by integrating in $(x,\xi) \in \Gamma \times \R^d$,
    \[
    \E_k(t) \leq \E_k(0) \, \exp \left(2\eta_{\lambda,V} \, t \right).
    \]
 Finally integrating in $k \in \Gamma^*$, we get
    \[
    \E(t) \leq \E(0) \, \exp \left(2\eta_{\lambda,V} \, t \right),
    \]
    and thus
    \[
    E_{\hbar,\lambda}(f(t),R(t))^2 \leq \E(0) \, \exp \left(2\eta_{\lambda,V} \, t \right).
    \]
    The theorem follows by minimizing the right-hand side as $Q^{\rm in}$ runs through $\C(f^{\rm in}, R^{\rm in})$.
\end{proof}

\section{Properties of periodic coherent states, Töplitz operator and Husimi density} \label{sec:propertiesSchroTopHus}

Recall that we introduced the coherent states $\coh{q}{p}$, $\Ell$-periodic coherent states $\percoh{q}{p}$, $\Ell$-periodic Töplitz operator $T_\Ell[\bullet]$ and $\Ell$-periodic Husimi density $\husimi[\bullet]$ in Definitions~\ref{def:periodizedCoherent}--~\ref{def:periodizedHusimi}.

\begin{lemma} [\textbf{Basic properties of the $\Ell$-periodic coherent states}] \label{lem:schroProp}
For any $(q,p) \in \R^d \times \R^d$, the $\Ell$-periodic coherent state $\percoh{q}{p}$ defined in~\eqref{eqdef:percohState} satisfy the following properties:
 \begin{itemize}
     \item denoting by $\left\{\coh{q}{p}_k \right\}_{k \in \Gamma^*}$ the Bloch decomposition of $\coh{q}{p}$, the coherent state defined in~\eqref{eqdef:cohState}, we have, for any $k \in \Gamma^*$,
     \begin{equation} \label{eq:BlochCoherent}
         \coh{q}{p}_k = \percoh{q}{p - \hbar k}.
     \end{equation}
          \item the $\Ell$-periodic coherent state $\percoh{q}{p}$ belongs to $L^2_{\rm per}(\Gamma)$, $q \mapsto \percoh{q}{p}$ is $\Ell$-periodic and  
     \begin{equation} \label{eq:percohNormlized}
     \fint_{\Gamma^*}\langle \underline{q,p - \hbar k} \percoh{q}{p - \hbar k}_{L^2_{\rm per}(\Gamma)} \, \dd k = 1.
     \end{equation}
 \end{itemize}
\end{lemma}

\begin{proof}
First note that the $\Ell$-periodic states are indeed well-defined, as
\[
\sum_{\ell \in \Ell} e^{-\frac{|\ell|^2}{2 \hbar}} < +\infty.
\]
Moreover, the $\Ell$-periodicity of $(q,p,y) \mapsto \percoh{q}{p}(y)$ in both the $q$ and $y$ variables is straightforward from~\eqref{eqdef:percohState} and the facts that $\coh{q}{p}(y + \ell) = \coh{q-\ell}{p}(y)$ and $\Ell = -\Ell$. Equation~\eqref{eq:BlochCoherent} is straightforward from the definition of the Bloch transform~\eqref{eqdef:BlochTransform}, as for any $q,p,y \in \R^d$ and $k \in \Gamma^*$ we have 
\[
\sum_{\ell \in \Ell} \coh{q}{p}(y + \ell) \, e^{-i k \cdot (y + \ell)} = \sum_{\ell \in \Ell} \coh{q}{p - \hbar k}(y + \ell) = \percoh{q}{p-\hbar k}(y).
\]
From the isometry property of the Bloch transform~\eqref{eq:BlochNormsEq} follows from~\eqref{eq:percohNormlized}, as
\[
\fint_{\Gamma^*}\langle \underline{q,p - \hbar k} \percoh{q}{p - \hbar k}_{L^2_{\rm per}(\Gamma)} \, \dd k = \langle q,p | q,p \rangle_{L^2(\R^d)} = 1.
\]
\end{proof}

\noindent In the following proposition, we show that the Husimi transform is indeed a periodized probability density.

\begin{proposition} [\textbf{Basic property of the Husimi density}] \label{prop:HusimiIsAProba}
Consider $R \in \ScEllone$ whose Bloch transform writes, for any $k \in \Gamma^*$,
\[
R_k = |\psi_k \rangle \langle \psi_k |
\]
with $\psi_k \in L^2_{\rm per}(\Gamma)$. Then $\husimi[R]$, the $\Ell$-periodic Husimi transform of $R$, belongs to $\mathcal{P}_\Ell$.
\end{proposition}

\begin{proof}
The density $\husimi[R]$ it obviously nonnegative and $\Ell$-periodic in the $q$-variable, we thus only have to show that it is a probability density. Fix $q \in \Gamma$ and $p \in \R^d$. By definition,
\[
\husimi[R](q,p) = \fint_{\Gamma^*} f_k(q,p) \, \dd k,
\]
with, for any $k \in \Gamma^*$, 
\[
f_k(q,p) = (\pi \hbar)^{-d/2} \left|(2 \pi \hbar)^{-d / 2}\sum_{\ell \in \Ell}\int_{\Gamma} e^{-\frac{|y+\ell-q|^2}{2 \hbar}} \, e^{-i (p-\hbar k) \cdot (y+\ell) / \hbar} \, \psi_k(y) \, \dd y\right|^2.
\]
As $\psi_k$ is $\Ell$-periodic, the above equals
\[
(\pi \hbar)^{-d/2} \left|(2 \pi \hbar)^{-d / 2}\sum_{\ell \in \Ell}\int_{\Gamma + \ell} e^{-\frac{|y-q|^2}{2 \hbar}} \, e^{-i (p-\hbar k) \cdot y / \hbar} \, \psi_k(y) \, \dd y\right|^2,
\]
that is
\[
(\pi \hbar)^{-d/2} \left|(2 \pi \hbar)^{-d / 2}\int_{\R^d} e^{-\frac{|y-q|^2}{2 \hbar}} \, e^{-i (p-\hbar k) \cdot y / \hbar} \, \psi_k(y) \, \dd y\right|^2.
\]
Integrating in $p \in \R^d$ and using Parseval's equality, we therefore have
\begin{equation} \label{eq:intf_kdp}
\int_{\R^d}f_k(q,p) \, \dd p = (\pi \hbar)^{-d / 2} \, \left\| y \mapsto e^{-\frac{|y-q|^2}{2 \hbar}} \, e^{i k\cdot y} \, \psi_k(y) \right\|_{L^2(\R^d)}^2 = (\pi \hbar)^{-d / 2} \, \int_{\R^d} e^{-\frac{|y-q|^2}{\hbar}} |\psi_k(y)|^2 \, \dd y.  
\end{equation}
Integrating in $q \in \Gamma$, it comes that
\[
\iint_{\Gamma \times \R^d}f_k(q,p) \, \dd q \, \dd p = (\pi \hbar)^{-d / 2} \, \int_{q \in \Gamma} \int_{y \in \R^d} e^{-\frac{|y-q|^2}{\hbar}} |\psi_k(y)|^2 \, \dd y \, \dd q.
\]
Changing variables $y-q \mapsto y$, the above becomes
\[
(\pi \hbar)^{-d / 2} \, \int_{q \in \Gamma} \int_{y \in \R^d} e^{-\frac{|y|^2}{\hbar}} |\psi_k(y+q)|^2 \, \dd y \, \dd q.
\]
Using Fubini's theorem and the fact that $\psi_k$ is $\Ell$-periodic, we henceforth have
\begin{equation} \label{eq:int_fkqp}
\iint_{\Gamma \times \R^d}f_k(q,p) \, \dd q \, \dd p = (\pi \hbar)^{-d / 2} \,  \int_{\R^d} e^{-\frac{|y|^2}{\hbar}} \left(\int_{\Gamma} |\psi_k(q)|^2 \,  \dd q \right)\dd y = \|\psi_k\|_{L^2_{\rm per}(\Gamma)}^2 = \mathrm{Tr}(R_k).
\end{equation}
We then conclude, as
\[
\fint_{\Gamma^*}\iint_{\Gamma \times \R^d}f_k(q,p)\, \dd q \, \dd p \, \dd k = \fint_{\Gamma^*} \mathrm{Tr}(R_k) \dd k = \perTr(R) = 1.
\]
\end{proof}

\subsection{Upper-bounds on pseudo-distance between $\Ell$-periodic Töplitz and Husimi related couples}

In this subsection, we provide an upper-bound on the classical/quantum pseudo-distance between $\Ell$-periodic Töplitz and Husimi related couples. The following proposition focuses on a couple composed of an $\Ell$-periodic Töplitz operator and its symbol.

\begin{proposition} [\textbf{Pseudo-distance between $\Ell$-periodic Töplitz operator and its symbol}] \label{prop:boundToplitz}
For each $f \in \mathcal{P}_\Ell$ with finite second moments, we have
\begin{equation} \label{eqprop:dist_f_Toplitz}
    E_{\hbar,\lambda} (f, \, T_\Ell [f]) \leq \sqrt{\frac{1 + \lambda^2}{2} \, d \, \hbar}.
\end{equation}
\end{proposition}

\begin{proof}
We define the diagonal coupling
\[
Q(x,\xi) \, \varphi (y) := f(x,\xi) \,  \, \fint_{\Gamma^*} e^{i k  \cdot y} \langle \underline{x,\xi - \hbar k} | \varphi_k \rangle_{L^2_{\rm per}(\Gamma)} \,  |\underline{x,\xi- \hbar k} \rangle(y)  \, \dd k. 
\]
From~\eqref{eq:BlochCoherent}--\eqref{eq:percohNormlized}, we have $Q \in \C(f,T_\Ell [f])$, and, by construction,
\[
\forall k \in \Gamma^*, \qquad Q_k(x,\xi)  = f(x,\xi) \,  |\underline{x,\xi- \hbar k} \rangle \langle \underline{x,\xi - \hbar k} |.   \phantom{\forall k \in \Gamma^*, \qquad}
\]
Therefore,
\begin{align*}
E_{\hbar,\lambda} (f, \, T_\Ell [f]))^2 &\leq  \iint_{\Gamma \times \R^d} \perTr \left( Q(x,\xi)^\frac12 \, c_\lambda(x,\xi) \, Q(x,\xi)^\frac12 \right) \, \dd x \, \dd \xi    \\
&=  \iint_{\Gamma \times \R^d} \fint_{\Gamma^*}    \mathrm{Tr}\left( Q_k(x,\xi)^\frac12 \, \big(c_\lambda(x,\xi) \big)_k \, Q_k(x,\xi)^\frac12 \right) \, \dd k \,  \dd x \, \dd \xi \\
&= \iint_{\Gamma \times \R^d}f(x,\xi) \fint_{\Gamma^*}    \langle \underline{x,\xi- \hbar k} | \, \hat{c}_\lambda(x,\xi - \hbar k) \, | \underline{x,\xi- \hbar k} \rangle_{L^2_{\rm per}(\Gamma)} \, \dd k \,  \dd x \, \dd \xi,
\end{align*}
where in the last equality we used~\eqref{eq:c_lba_k} in Lemma~\ref{lem:BlochTransFormulas}. Since by~\eqref{eq:equivalenceThetaPGammasqr},
\begin{align*}
   \hat{c}_\lambda (x,\xi) \equiv \lambda^2 \Theta(|P_{\Gamma}(y-x)|^2) + (\xi + i \hbar \nabla_y)^2 \leq \lambda^2 |P_{\Gamma}(y-x)|^2 + (\xi + i \hbar \nabla_y)^2,
\end{align*}
we conclude that
\begin{multline*}
E_{\hbar,\lambda} (f, \, T_\Ell [f]))^2 \\\leq \iint_{\Gamma \times \R^d}f(x,\xi) \fint_{\Gamma^*}  \langle \underline{x,\xi- \hbar k} | \,\lambda^2 |P_{\Gamma}(y-x)|^2 + (\xi + i \hbar \nabla_y)^2   \, | \underline{x,\xi- \hbar k} \rangle_{L^2_{\rm per}(\Gamma)} \, \dd k \,  \dd x \, \dd \xi.
\end{multline*}
On the one hand, for any $(x,\xi) \in \R^d \times \R^d$, the (non-periodic) coherent state $| x,\xi \rangle$ satisfies
\[
(\xi + i \hbar \nabla_y) | x,\xi \rangle (y) = \left(\xi + i \hbar \left[ - \frac{y-x}{\hbar} + \frac{i}{\hbar} \xi \right] \right) |x,\xi \rangle (y ) = - i \left( y-x \right) |x,\xi \rangle (y),
\]
therefore
\[
(\xi + i \hbar \nabla_y)^2 | x,\xi \rangle (y) = \left( d \hbar- |y-x|^2\right) |x,\xi \rangle (y). 
\]
We deduce by linearity that the $\Ell$-periodic coherent state $\percoh{x}{\xi}$ satisfies
\begin{align*}
(\xi + i \hbar \nabla_y)^2 | \underline{x,\xi} \rangle (y) &= \sum_{\ell \in \Ell} (\xi + i \hbar \nabla_y)^2 | x,\xi \rangle (y) = \sum_{\ell \in \Ell} \left( d \hbar- |y + \ell -x|^2\right) |x,\xi \rangle (y + \ell).
\end{align*}
On the other hand,
\[
\lambda^2 |P_{\Gamma}(y-x)|^2 \, \percoh{x}{\xi} (y) = \lambda^2 \sum_{\ell \in \Ell} |P_{\Gamma}(y-x)|^2 \, \coh{x}{\xi} (y + \ell),
\]
so that, as $P_{\Gamma}(y-x) = P_{\Gamma}(y + \ell -x)$,
\begin{multline*}
(\,\lambda^2 |P_{\Gamma}(y-x)|^2 + (\xi + i \hbar \nabla_y)^2  ) \, \percoh{x}{\xi} (y) \\
= \sum_{\ell \in \Ell} \left(d \hbar  + \lambda^2 |P_{\Gamma}(y + \ell -x)|^2 - |y + \ell -x|^2 \right)\, \coh{x}{\xi} (y + \ell).
\end{multline*}
It comes in particular that, for any $k \in \Gamma^*$,
\[
(\lambda^2 |P_{\Gamma}(y-x)|^2 + (\xi + i \hbar \nabla_y)^2  )\, \percoh{x}{\xi - \hbar k} = \left(y \mapsto (d \hbar  + \lambda^2 |P_{\Gamma}(y -x)|^2 - |y  -x|^2 )\, \coh{x}{\xi}(y) \right)_k,
\]
where the notation $(\bullet)_k$ stands for the $k^{th}$ component of the Bloch decomposition of the function $\bullet$. Since $\percoh{x}{\xi - \hbar k} = (\coh{x}{\xi})_k$ by~\eqref{eq:BlochCoherent}, we have, using the isometry property of the Bloch transform~\eqref{eq:BlochNormsEq},
\begin{multline*}
\fint_{\Gamma^*}  \langle \underline{x,\xi- \hbar k} | \,\lambda^2 |P_{\Gamma}(y-x)|^2 + (\xi + i \hbar \nabla_y)^2   \, | \underline{x,\xi- \hbar k} \rangle_{L^2_{\rm per}(\Gamma)} \, \dd k \\
= \langle x,\xi |  (d \hbar  + \lambda^2 |P_{\Gamma}(y -x)|^2 - |y  -x|^2 )\, \coh{x}{\xi}_{L^2(\R^d)},
\end{multline*}
which, as $|P_{\Gamma}(y -x)| \leq |y - x|$, is less than
\[
d \hbar + \langle x,\xi |  ( \lambda^2 - 1)  |y  -x|^2 \, \coh{x}{\xi}_{L^2(\R^d)} = d \hbar + (\lambda^2 - 1) \frac{d \hbar}{2} = (1 + \lambda^2)\frac{d \hbar}{2},
\]
hence allowing to conclude to~\eqref{eqprop:dist_f_Toplitz}.
\end{proof}

\medskip

\noindent In the following proposition, we provide an upper-bound on the classical/quantum pseudo-distance between an operator in $\ScEllone$ for which all Bloch components are rank-$1$ projectors (``periodic pure state'') and its associated Husimi density.

\begin{proposition} [\textbf{Pseudo-distance between ``periodic pure state'' and its Husimi density}] \label{prop:EhbarHusimi}
    Consider $R \in \ScEllone$ whose Bloch transform writes, for any $k \in \Gamma^*$,
\[
R_k = |\psi_k \rangle \langle \psi_k |,
\]
with $\psi_k \in H^2_{\rm per}(\Gamma)$. Then
    \begin{equation} \label{eq:EhbarHusimi}
        E_{\hbar, 1}(\husimi[R],R) \leq \sqrt{d \hbar \, \mathbf{c}_{\psi} + 2 \,\Delta_{\Gamma,\hbar}^2(R)},
    \end{equation}
where $\mathbf{c}_\bullet$ and $\Delta_{\Gamma,\hbar}$ are defined respectively  in~\eqref{eqdef:cBold} and~\eqref{eqdef:stdDevR}.
\end{proposition}

\begin{proof}
Recall that for any $(q,p) \in \R^d \times \R^d$,
\[
\husimi[R](q,p) := \fint_{\Gamma^*} f_k(q,p) \, \dd k,
\]
where
\[
 f_k(q,p) := (2 \pi \hbar)^{-d} \, \left|\langle\underline{q,p-\hbar k} | \psi_k \rangle_{L^2_{\rm per}(\Gamma)} \right|^2.
\]
We let $Q$ such that for any $k \in \Gamma^*$, the $k^{th}$ component of $Q$'s Bloch transform, $Q_k$, is defined by
\[
Q_k(q,p) := f_k(q,p) \, R_k.
\]
We let
\[
\mathcal{E}_k := \iint_{\Gamma \times \R^d} \mathrm{Tr} \left(Q_k(q,p)^{\frac12} \hat{c}_1(q,p - \hbar k)  Q_k(q,p)^{\frac12} \right) \, \dd q \, \dd p,
\]
which can be rewritten as
\[
\mathcal{E}_k = \iint_{\Gamma \times \R^d} f_k(q,p) \, \langle \psi_k | \hat{c}_1(q,p - \hbar k)  | \psi_k \rangle_{L^2_{\rm per}(\Gamma)}\, \dd q \, \dd p,
\]
or, changing variables $p - \hbar k \mapsto p$,
\[
\mathcal{E}_k = \iint_{\Gamma \times \R^d} f_k(q,p + \hbar k) \, \langle \psi_k | \hat{c}_1(q,p)  | \psi_k \rangle_{L^2_{\rm per}(\Gamma)} \, \dd q \, \dd p.
\]
Recall that
\[
\hat{c}_1 (q,p) = \Theta(|P_{\Gamma}(y-q)|^2) + (p + i \hbar \nabla_y)^2 \leq |P_{\Gamma}(y-q)|^2 + (p + i \hbar \nabla_y)^2,
\]
we denote
\[
\E_k = \E_{k, 1} + \E_{k, \, 2},
\]
with
\[
\E_{k, 1} := \iint_{\Gamma \times \R^d} f_k(q,p + \hbar k) \left(\int_{\Gamma} |P_\Gamma(y-q)|^2 \, |\psi_k(y)|^2 \, \dd y \right) \dd q \, \dd p
\]
and
\[
\E_{k, 2} := \iint_{\Gamma \times \R^d} f_k(q,p + \hbar k) \left(\int_{\Gamma} \overline{\psi_k(y)} \, (p + i \hbar \nabla_y)^2 \,\psi_k(y) \,  \dd y \right) \dd q \, \dd p.
\]
$\bullet$ \textbf{Upper bound for} $\E_{k,1}$.

\medskip

\noindent By Fubini's theorem, we have
\begin{align*}
 \E_{k, 1} = \iint_{\Gamma \times \Gamma}  \left(\int_{\R^d} f_k(q,p)\, \dd p \right) |P_\Gamma(y-q)|^2 \, |\psi_k(y)|^2 \, \dd y \, \dd q.
\end{align*}
Recalling Equation~\eqref{eq:intf_kdp}, the above writes
\[
\E_{k, 1} = \iint_{\Gamma \times \Gamma}  \left(  (\pi \hbar)^{-d / 2} \, \int_{\R^d} e^{-\frac{|z-q|^2}{\hbar}} |\psi_k(z)|^2 \, \dd z \right) |P_\Gamma(y-q)|^2 \, |\psi_k(y)|^2 \, \dd y \, \dd q.
\]
We change variables $z \mapsto x = z-q$ and obtain
\[
\E_{k, 1} = \iint_{\Gamma \times \Gamma}  \left(  (\pi \hbar)^{-d / 2} \, \int_{\R^d} e^{-\frac{|x|^2}{\hbar}} |\psi_k(x+q)|^2 \, \dd x \right) |P_\Gamma(y-q)|^2 \, |\psi_k(y)|^2 \, \dd y \, \dd q.
\]
By Fubini's theorem,
\[
\E_{k, 1} = (\pi \hbar)^{-d / 2} \, \int_{x \in \R^d}\int_{y \in \Gamma}   e^{-\frac{|x|^2}{\hbar}} \left(\int_{\Gamma}|\psi_k(x+q)|^2 |P_\Gamma(y-q)|^2  \, \dd q\right) |\psi_k(y)|^2 \, \dd x \, \dd y.
\]
Since the function $q \mapsto |\psi_k(x+q)|^2 |P_\Gamma(y-q)|^2$ is $\Ell$-periodic, we have
\[
\int_{\Gamma}|\psi_k(x+q)|^2 |P_\Gamma(y-q)|^2  \, \dd q = \int_{\Gamma}|\psi_k(q)|^2 |P_\Gamma(y+x-q)|^2  \, \dd q,
\]
so that
\[
\E_{k, 1} = (\pi \hbar)^{-d / 2} \, \int_{x \in \R^d}\int_{y \in \Gamma}   e^{-\frac{|x|^2}{\hbar}} \left(\int_{\Gamma}|\psi_k(q)|^2 |P_\Gamma(y+x-q)|^2  \, \dd q\right) |\psi_k(y)|^2 \, \dd x \, \dd y,
\]
hence
\[
\E_{k, 1} =  \iint_{\Gamma \times \Gamma}  \left((\pi \hbar)^{-d / 2} \,\int_{\R^d} |P_\Gamma(y+x-q)|^2  \,  e^{-\frac{|x|^2}{\hbar}} \dd x\right) |\psi_k(q)|^2 |\psi_k(y)|^2 \, \dd y \, \dd q,
\]
Now, denoting $\ell_{y-q} = y-q -  P_\Gamma(y-q) \in \Ell$, notice that
\begin{multline*}
|P_\Gamma(y+x-q)|^2  = \min_{\ell \in \Ell}|y + x - q - \ell|^2  \leq |x+ y-q - \ell_{y-q}|^2 =  |x+P_\Gamma(y-q)|^2 \\
= |x|^2 + |P_\Gamma(y-q)|^2 + 2 \, x \cdot P_\Gamma(y-q).
\end{multline*}
Since
\[
(\pi \hbar)^{-d/2} \,\int_{\R^d} |x|^2  \,  e^{-\frac{|x|^2}{\hbar}} \dd x = \frac{d \hbar}{2}, \quad \text{and} \quad (\pi \hbar)^{-d/2} \,\int_{\R^d} x  \,  e^{-\frac{|x|^2}{\hbar}} \dd x = 0,
\]
we henceforth obtain
\begin{align*}
\E_{k,1} &\leq \iint_{\Gamma \times \Gamma}  \left(\frac{d \hbar}{2}\right) |\psi_k(q)|^2 |\psi_k(y)|^2 \, \dd y \, \dd q +  \iint_{\Gamma \times \Gamma}  |P_\Gamma(y-q)|^2 \, |\psi_k(q)|^2 |\psi_k(y)|^2 \, \dd y \, \dd q \\
&= \frac{d \hbar}{2} \, \|\psi_k\|^4_{L^2_{\rm per}(\Gamma)} + \iint_{\Gamma \times \Gamma}  |P_\Gamma(y-q)|^2 \, |\psi_k(q)|^2 |\psi_k(y)|^2 \, \dd y \, \dd q.
\end{align*}
$\bullet$ \textbf{Upper bound for} $\E_{k,2}$. Let us decompose
\[
(p + i\hbar \nabla_y)^2 = |p|^2 - 2  p  \cdot  (-i\hbar\nabla_y) + (-i\hbar \nabla_y)^2,
\]
we have, recalling~\eqref{eq:int_fkqp},
\begin{multline*}
    \E_{k,2} = \|\psi_k\|^2_{L^2_{\rm per}(\Gamma)}\iint_{\Gamma \times \R^d} |p|^2 \, f_k(q,p + \hbar k) \, \dd q \, \dd p +  \hbar^2 \; \|\psi_k\|^2_{L^2_{\rm per}(\Gamma)}\int_{\Gamma}  |\nabla \psi_k(y)|^2 \,  \dd y  \\
    - 2 \iint_{\Gamma \times \R^d} p \, f_k(q,p + \hbar k) \, \dd q \, \dd p  \cdot \int_{\Gamma} \overline{\psi_k(y)} \, (-i \hbar \nabla_y) \,\psi_k(y) \,  \dd y.
\end{multline*}
Denote by $\hat{g}$ the the Fourier transform of
\[
g : y \mapsto e^{-\frac{|y-q|^2}{2 \hbar}} \, \psi_k(y),
\]
we have
\[
f_k(q,p + \hbar k) = (\pi \hbar)^{-d/2} \left|  \hbar^{-d/2} \,\hat{g} \left( \frac{p}{\hbar}\right) \right|^2.
\]
On the one hand,
\[
\int_{\R^d} |p|^2 \, f_k(q,p + \hbar k) \, \dd p = (\pi \hbar)^{-d/2} \int_{\R^d} \left| i p \,  \hbar^{-d/2} \,\hat{g} \left( \frac{p}{\hbar}\right) \right|^2 \, \dd p = (\pi \hbar)^{-d/2} \, \hbar^2 \int_{\R^d} \left| i p  \,\hat{g} \left(p\right) \right|^2 \, \dd p.
\]
By Parseval's equality, the above transforms into
\begin{align*}
\hbar^2 \left( (\pi \hbar)^{-d/2} \int_{\R^d} \left| \nabla g(y)  \right|^2 \, \dd y\right)&= \hbar^2 \; (\pi \hbar)^{-d/2} \int_{\R^d} \left| \nabla \psi_k(y)  - \frac{y-q}{\hbar}\, \psi_k(y)\right|^2 \, e^{-\frac{|y-q|^2}{\hbar}}\, \dd y \\
&= \hbar^2 \;(\pi \hbar)^{-d/2} \int_{\R^d} \left| \nabla \psi_k(y+q)  - \frac{y}{\hbar}\, \psi_k(y+q)\right|^2 \, e^{-\frac{|y|^2}{\hbar}}\, \dd y.
\end{align*}
Integrating in $q \in \Gamma$, we obtain
\begin{align*}
\iint_{\Gamma \times \R^d} |p|^2 \, f_k(q,p + \hbar k) \, \dd q \, \dd p &= \hbar^2  (\pi \hbar)^{-d/2} \int_{\R^d}\left( \int_{\Gamma}\left| \nabla \psi_k(y+q) \right|^2 \, \dd q \right) e^{-\frac{|y|^2}{\hbar}}\, \dd y \\
&+ \hbar^2  (\pi \hbar)^{-d/2} \int_{\R^d}\left( \int_{\Gamma} |\psi_k(y+q)|^2 \dd q \right) \left|\frac{y}{\hbar} \right|^2 e^{-\frac{|y|^2}{\hbar}}\, \dd y   \\
&- 2\hbar^2  (\pi \hbar)^{-d/2} \int_{\R^d}\, \mathfrak{Re}\left( \int_{\Gamma} \psi_k(y+q)\nabla \psi_k(y+q)  \, \dd q \right) \cdot \frac{y}{\hbar} \, e^{-\frac{|y|^2}{\hbar}}\, \dd y.
\end{align*}
Since $\psi_k \in H^1_{\rm per} (\Gamma)$, we have
\[
 \int_{\Gamma}\left| \nabla \psi_k(y+q) \right|^2 \, \dd q  =  \int_{\Gamma}\left| \nabla \psi_k(q) \right|^2 \, \dd q = \|\nabla\psi_k\|^2_{L^2_{\rm per}(\Gamma)},
\]
as well as
\[
\int_{\Gamma} |\psi_k(y+q)|^2 \dd q = \int_{\Gamma} |\psi_k(q)|^2 \dd q = \|\psi_k\|^2_{L^2_{\rm per}(\Gamma)},
\]
and
\[
 \int_{\Gamma} \psi_k(y+q)\nabla \psi_k(y+q)  \, \dd q  =  \int_{\Gamma} \psi_k(q) \nabla \psi_k(q)  \, \dd q.
\]
Combining these with
\[
(\pi \hbar)^{-d/2} \int_{\R^d} e^{-\frac{|y|^2}{\hbar}}\, \dd y = 1, \quad (\pi \hbar)^{-d/2} \int_{\R^d} y \, e^{-\frac{|y|^2}{\hbar}}\, \dd y = 0 \; \; \text{and} \; \;  (\pi \hbar)^{-d/2} \int_{\R^d} |y|^2 \, e^{-\frac{|y|^2}{\hbar}}\, \dd y = \frac{d \hbar}{2},
\]
we conclude that
\[
\iint_{\Gamma \times \R^d} |p|^2 \, f_k(q,p + \hbar k) \, \dd q \, \dd p = \hbar^2 \|\nabla\psi_k\|^2_{L^2_{\rm per}(\Gamma)} + \frac{d \hbar}{2} \|\psi_k\|^2_{L^2_{\rm per}(\Gamma)}.
\]
On the other hand, 
\begin{multline*}
 \int_{\R^d} p \, f_k(q,p + \hbar k) \, \dd p = (\pi \hbar)^{-d/2} \int_{\R^d} p \left|  \hbar^{-d/2} \,\hat{g} \left( \frac{p}{\hbar}\right) \right|^2 \, \dd p = (\pi \hbar)^{-d/2} \, \int_{\R^d} p\left|\hat{g} \left(p\right) \right|^2 \, \dd p \\
 = -i\hbar \,  (\pi \hbar)^{-d/2} \int_{\R^d} \overline{\hat{g} (p) } \, \left(ip \, \hat{g} (p) \right)\, \dd p.
\end{multline*}
The isometry property of the Fourier transform yields
\begin{align*}
\langle \hat{g} \, | \widehat{\nabla g}\, \rangle_{L^2(\R^d)}  = \langle g \, | \nabla g\, \rangle_{L^2(\R^d)},
\end{align*}
so that we have
\begin{align*}
\int_{\R^d} p \, f_k(q,p + \hbar k) \, \dd p &=  -i\hbar \,  (\pi \hbar)^{-d/2} \int_{\R^d} \overline{\psi_k(y)} \left( \nabla \psi_k(y)  - \frac{y-q}{\hbar} \, \psi_k(y)\right) e^{-\frac{|y-q|^2}{\hbar}}  \, \dd y \\
&= -i\hbar \,  (\pi \hbar)^{-d/2} \int_{\R^d} \overline{\psi_k(y+q)} \left( \nabla \psi_k(y+q)  - \frac{y}{\hbar} \, \psi_k(y+q)\right) e^{-\frac{|y|^2}{\hbar}}  \, \dd y.
\end{align*}
Again, integrating in $q \in \Gamma$ and by $\Ell$-periodicity of $\psi_k$, we obtain
\begin{multline*}
\iint_{\Gamma \times\R^d} p \, f_k(q,p + \hbar k) \, \dd q \,  \dd p
= -i\hbar \,  (\pi \hbar)^{-d/2} \int_{\R^d} \left(\int_{\Gamma}\overline{\psi_k(q)} \nabla \psi_k(q) \, \dd q \right) e^{-\frac{|y|^2}{\hbar}}  \, \dd y\\
+ i\hbar \,  (\pi \hbar)^{-d/2} \left(\int_{\R^d} |\psi_k(q)|^2 \, \dd q\right) \frac{y}{\hbar} e^{-\frac{|y|^2}{\hbar}}  \, \dd y,
\end{multline*}
that is
\[
\langle \psi_k | \, -i\hbar \nabla \psi_k \rangle_{L^2_{\rm per}(\Gamma)}.
\]
All in all,
\begin{align*}
\E_{k,2} &= \frac{d \hbar}{2} \|\psi_k\|^4_{L^2_{\rm per}(\Gamma)} + 2  \|\psi_k\|^2_{L^2_{\rm per}(\Gamma)} \|- i \hbar\nabla \psi_k\|^2_{L^2_{\rm per}(\Gamma)}  - 2 \langle \psi_k | \, -i\hbar \nabla \psi_k \rangle_{L^2_{\rm per}(\Gamma)}^2 \\
&= \frac{d \hbar}{2} \|\psi_k\|^4_{L^2_{\rm per}(\Gamma)}  + \iint_{\Gamma \times \Gamma} (-i \hbar \nabla_y + i \hbar \nabla_q)^2 \, |\psi_k(y)|^2 \, |\psi_k(q)|^2 \, \dd y \, \dd q.
\end{align*}
$\bullet$ \textbf{Conclusion to the proof.}

\medskip

\noindent Combining the upper bounds on $\E_{k,1}$ and $\E_{k,2}$, we get
\[
\E_k = \E_{k,1} + \E_{k,2} \leq d \hbar \, \|\psi_k\|^4_{L^2_{\rm per}(\Gamma)} +\iint_{\Gamma \times \Gamma} \left[ |P_\Gamma(y-q)|^2 + (i \hbar \nabla_y - i \hbar \nabla_q)^2 \right] \, |\psi_k(y)|^2 \, |\psi_k(q)|^2 \, \dd y \, \dd q.
\]
Taking the average in $k \in \Gamma^*$, we conclude to the proposition, as
\[
E_{\hbar, 1}(\husimi[R], \, R) \leq \fint_{\Gamma^*} \E_k \, \dd k = d \hbar \, \mathbf{c}_{\psi} + 2\Delta^2_{\Gamma,\hbar} (R).
\]
\end{proof}

\section{Proof of Theorem~\ref{theo:obsfinal}}

In this final section, we provide the proof of our core observability Theorem~\ref{theo:obsfinal}, which relies mainly on the Liouville-von Neumann stability result (Theorem~\ref{theo:stability}), bounds related to the Töplitz and Husimi transforms (Propositions~\ref{prop:boundToplitz} and~\ref{prop:EhbarHusimi}) and the classical observability result (Lemma~\ref{lem:liouvilleobs}).

\begin{proof}[Proof of Theorem~\ref{theo:obsfinal}] \label{sec:proof}
Using the notations of Theorem~\ref{theo:obsfinal}, we can write
\begin{equation*}
R(t) = U(t)^* \, R^{\rm in} \, U(t),
\end{equation*}
and set, for some arbitrary $f^{\rm in} \in \mathcal{P}_\Ell$,
\[
f(t, X, \Xi):=f^{\rm in} \left( \Phi_{-t}(X, \Xi) \right) \quad \text{for a.e. } (X,\Xi) \in \R^d \times \R^d,
\]
where we recall that $U(t)$ and $\Phi_t$ are respectively the quantum and classical flows (see Subsection~\ref{subsec:flowsBlochflows}) associated to the potential $V$ (and $\hbar$). 

\medskip

\noindent Let $0 \leq \chi \in \mathcal{C}^0(\R^d)$ a Lipschitz and $\Ell$-periodic function. We have, for all $Q \equiv Q(t,x,\xi) \in \C(f(t),R(t))$,
\begin{equation*}
    \perTr(\chi R(t)) - \iint_{\Gamma \times \R^d} \chi (x) f(t,x,\xi) \, \dd x \, \dd \xi = \iint_{\Gamma \times \R^d} \perTr \left((\chi (y) - \chi (x)) Q(t,x,\xi) \right) \dd x \, \dd \xi,
\end{equation*}
as
\[
\perTr(\chi R(t)) = \perTr \left(\chi(y) \iint_{\Gamma \times \R^d} Q(t,x,\xi) \, \dd x \, \dd \xi \right) =  \iint_{\Gamma \times \R^d} \perTr\left(\chi(y)  Q(t,x,\xi)  \right) \dd x \, \dd \xi,
\]
and
\[
\iint_{\Gamma \times \R^d} \chi (x) f(t,x,\xi) \, \dd x \, \dd \xi = \iint_{\Gamma \times \R^d} \chi (x) \perTr\left( Q(t,x,\xi)  \right) \, \dd x \, \dd \xi =  \iint_{\Gamma \times \R^d} \perTr\left(\chi(x)  Q(t,x,\xi)  \right) \dd x \, \dd \xi.
\]
It comes that
\begin{align*}
    \left| \perTr(\chi R(t)) - \iint_{\Gamma \times \R^d} \chi (x) f(t,x,\xi) \, \dd x \, \dd \xi \right| &\leq \iint_{\Gamma \times \R^d} \left| \perTr\left( (\chi (y) - \chi (x)) \, Q(t,x,\xi)  \right)  \right| \dd x \, \dd \xi \\
    &= \iint_{\Gamma \times \R^d} \left|  \fint_{\Gamma^*} \mathrm{Tr} \left( (\chi (y) - \chi (x)) \, Q(t,x,\xi)_k  \right) \dd k \right| \dd x \, \dd \xi
\end{align*}
where the last equality comes by remarking that the multiplication by an $\Ell$-periodic function commutes with the Bloch transform. The above term is therefore smaller than
\begin{align*}
 &\phantom{\leq} \fint_{\Gamma^*}\iint_{\Gamma \times \R^d} \left| \mathrm{Tr} \left(  Q(t,x,\xi)^{\frac12}_k  \, (\chi (y) - \chi (x)) \, Q(t,x,\xi)^{\frac12}_k  \right)  \right| \dd k \, \dd x \, \dd \xi \\
  &\leq \fint_{\Gamma^*} \iint_{\Gamma \times \R^d} \mathrm{Tr} \left(  Q(t,x,\xi)^{\frac12}_k  \, |\chi (y) - \chi (x)| \, Q(t,x,\xi)^{\frac12}_k  \right)  \dd k \, \dd x \, \dd \xi \\
    &\leq \mathrm{Lip}(\chi) \fint_{\Gamma^*} \iint_{\Gamma \times \R^d} \mathrm{Tr} \left(  Q(t,x,\xi)^{\frac12}_k  \, |P_{\Gamma}(x-y)| \, Q(t,x,\xi)^{\frac12}_k  \right)  \dd k \, \dd x \, \dd \xi,
\end{align*}
the last inequality coming from the fact that $\chi$ is $\Ell$-periodic. As $z \leq \frac12 \left(\varepsilon z^2 + \e^{-1} \right)$ for any $z, \varepsilon > 0$, the above integral is less than
\begin{multline*}
\frac{\e}{2} \fint_{\Gamma^*}\iint_{\Gamma \times \R^d} \mathrm{Tr} \left(  Q(t,x,\xi)^{\frac12}_k  \, |P_{\Gamma}(x-y)|^2 \, Q(t,x,\xi)^{\frac12}_k  \right)  \dd k \, \dd x \, \dd \xi \\
+ \frac{\e^{-1}}{2}  \fint_{\Gamma^*}\iint_{\Gamma \times \R^d} \mathrm{Tr} \left(  Q(t,x,\xi)^{\frac12}_k  Q(t,x,\xi)^{\frac12}_k  \right)  \dd k \, \dd x \, \dd \xi,
\end{multline*}
that is exactly
\begin{equation} \label{eqproof:epsilon}
\frac{\e}{2} \fint_{\Gamma^*}\iint_{\Gamma \times \R^d} \mathrm{Tr} \left(  Q(t,x,\xi)^{\frac12}_k  \, |P_{\Gamma}(x-y)|^2 \, Q(t,x,\xi)^{\frac12}_k  \right)  \dd k \, \dd x \, \dd \xi + \frac{\e^{-1}}{2},
\end{equation}
as
\[
\fint_{\Gamma^*}\iint_{\Gamma \times \R^d} \mathrm{Tr} \left(  Q(t,x,\xi)_k  \right)  \dd k \, \dd x \, \dd \xi =  \fint_{\Gamma^*}\mathrm{Tr} \left[\left(  \iint_{\Gamma \times \R^d}  Q(t,x,\xi) \, \dd x \, \dd \xi  \right)_k \right] \dd k = \perTr(R(t)) = 1.
\]
Minimizing~\eqref{eqproof:epsilon} in $\e$, we conclude that
\begin{multline*}
    \left| \perTr(\chi R(t)) - \iint_{\Gamma \times \R^d} \chi (x) f(t,x,\xi) \, \dd x \, \dd \xi \right| \\
    = \mathrm{Lip}(\chi) \left(\fint_{\Gamma^*}\iint_{\Gamma \times \R^d} \mathrm{Tr} \left(  Q(t,x,\xi)^{\frac12}_k  \, |P_{\Gamma}(x-y)|^2 \, Q(t,x,\xi)^{\frac12}_k  \right)  \dd k \, \dd x \, \dd \xi \right)^{\frac12},
\end{multline*}
which, as $|P_{\Gamma}(x-y)|^2 \leq \frac{2 \gamma_+}{\gamma_-} \Theta(|P_\Gamma(x-y)|^2)$ and $0 \leq (\xi - \hbar k + i \hbar \nabla_y)^2$, is itself less than
\[
\sqrt{\frac{2 \gamma_+}{\gamma_-}}\frac{\mathrm{Lip}(\chi)}{\lambda} \left(\fint_{\Gamma^*}\iint_{\Gamma \times \R^d} \mathrm{Tr} \left(  Q(t,x,\xi)^{\frac12}_k  \, \hat{c}_\lambda(x,\xi - \hbar k) \, Q(t,x,\xi)^{\frac12}_k  \right)  \dd k \, \dd x \, \dd \xi \right)^{\frac12}.
\]
Recalling~\eqref{eq:c_lba_k} in Lemma~\ref{lem:BlochTransFormulas} as well as Corollary~\ref{cor:fundamentalPropertyBlochOp} in Appendix~\ref{sec:Bloch}, we obtain that the $k^{th}$ term of the Bloch transform of $Q(t,x,\xi)^{\frac12}  \, c_{\lambda}(x,\xi) \, Q(t,x,\xi)^{\frac12}$ exactly writes $Q(t,x,\xi)^{\frac12}_k  \, \hat{c}_\lambda(x,\xi - \hbar k) \, Q(t,x,\xi)^{\frac12}_k$. Hence the above terms equals
\[
\sqrt{\frac{2 \gamma_+}{\gamma_-}}\frac{\mathrm{Lip}(\chi)}{\lambda} \left(\iint_{\Gamma \times \R^d} \perTr \left(  Q(t,x,\xi)^{\frac12}  \, c_{\lambda}(x,\xi) \, Q(t,x,\xi)^{\frac12}  \right)  \dd x \, \dd \xi \right)^{\frac12},
\]
allowing, by minimizing as $Q(t,\cdot,\cdot)$ runs through $\C(f(t),R(t))$, to conclude to
\begin{equation*}
\left| \perTr(\chi R(t)) - \iint_{\Gamma \times \R^d} \chi (x) f(t,x,\xi) \, \dd x \, \dd \xi \right| \leq \sqrt{\frac{2 \gamma_+}{\gamma_-}} \frac{\mathrm{Lip}(\chi)}{\lambda}  \, E_{\hbar,\lambda} (f(t),R(t)).
\end{equation*}
By Theorem~\ref{theo:stability}, we obtain
\begin{equation*}
\left| \perTr(\chi R(t)) - \iint_{\Gamma \times \R^d} \chi (x) f(t,x,\xi) \, \dd x \, \dd \xi \right| \leq  {\sqrt{\frac{2 \gamma_+}{\gamma_-}}} \, \frac{\mathrm{Lip}(\chi)}{\lambda} \, \exp \left(\eta_{\lambda,V} \, t \right) \, E_{\hbar,\lambda} (f^{\rm in}, \, R^{\rm in}),
\end{equation*}
where we recall the notation $\eta_{\lambda,V} =  {\frac{\gamma_+}{\gamma_-}} \left(\lambda + \frac{\mathrm{Lip}(\nabla V)^2}{\lambda} \right)$. On the other hand, denoting $\Phi_t(x,\xi)$ by $(X(t;x,\xi), \, \Xi(t;x,\xi))$ and recalling that $f^{\rm in}(\cdot,\xi)$ and $\chi$ are $\Ell$-periodic, we apply Lemma~\ref{lem:LiouvilleChangeVar} to get
\begin{align}
    \iint_{\Gamma \times \R^d} \chi (x) f(t,x,\xi) \, \dd x \, \dd \xi &= \iint_{\Gamma \times \R^d} \chi (x) f^{\rm in}(X(-t;x,\xi), \, \Xi(-t;x,\xi)) \, \dd x \, \dd \xi \nonumber\\
    &= \iint_{\Gamma \times \R^d} \chi (X(t;x,\xi)) f^{\rm in}(x, \xi) \, \dd x \, \dd \xi. \label{eq:QuestionChi}
\end{align}
Therefore,
\begin{multline*}
    \int_0^T \perTr (\chi R(t)) \, \dd t \geq \iint_{\Gamma \times \R^d} \left( \int_0^T \chi (X(t;x,\xi)) \, \dd t \right)  f^{\rm in}(x, \xi) \, \dd x \, \dd \xi \\
    -  {\sqrt{\frac{2 \gamma_+}{\gamma_-}}} \,\frac{\mathrm{Lip}(\chi)}{\lambda} \, \left(\int_0^T \exp  \left(\eta_{\lambda,V} \, t \right) \dd t \right) E_{\hbar,\lambda} (f^{\rm in}, \, R^{\rm in}),
\end{multline*}
which, by direct integration, becomes
\begin{multline} \label{eq:intermedproof0}
    \int_0^T \perTr (\chi R(t)) \, \dd t \geq \iint_{\Gamma \times \R^d} \left( \int_0^T \chi (X(t;x,\xi)) \, \dd t \right)  f^{\rm in}(x, \xi) \, \dd x \, \dd \xi \\
    -  {\sqrt{\frac{2 \gamma_+}{\gamma_-}}} \,\frac{\mathrm{Lip}(\chi)}{\lambda} \, \left(\frac{\exp  \left(\eta_{\lambda,V} \, T \right) - 1 }{\eta_{\lambda,V}} \right) \, E_{\hbar,\lambda} (f^{\rm in}, \, R^{\rm in}).
\end{multline}
We now choose, for some $\delta > 0$,
\[
\chi(x) = \left(1 - \frac{\mathrm{dist}(x,\Omega^\Ell)}{\delta} \right)_+,
\]
which is Lipzchitz with $\mathrm{Lip}(\chi) = \delta^{-1}$ and $\Ell$-periodic (as $\Omega^\Ell$ is $\Ell$-periodic), and satisfies
\begin{equation} \label{eqproof:boundsChi}
    \indic_{\Omega^\Ell_\delta} \geq \chi \geq \indic_{\Omega^\Ell} \geq  \indic_{\Omega}.
\end{equation}
Combining~\eqref{eq:intermedproof0} with~\eqref{eqproof:boundsChi} yields
\begin{multline*}
    \int_0^T \underline{\text{Tr}} \left(\indic_{\Omega_\delta^\Ell} \,  R(t) \, \indic_{\Omega_\delta^\Ell}\right) \, \text{d} t \geq \iint_{\Gamma \times \R^d} \left( \int_0^T \indic_{\Omega} (X(t;x,\xi)) \, {\text{d}} t \right)  f^{\text{in}}(x, \xi) \, {\text{d}} x \, {\text{d}} \xi \\
    -  {\sqrt{\frac{2 \gamma_+}{\gamma_-}}} \left(\frac{1/\delta}{\lambda} \right) \left(\frac{\exp  \left(\eta_{\lambda,V} \, T \right) - 1 }{\eta_{\lambda,V}} \right) \, E_{\hbar,\lambda} (f^{\text{in}}, \, R^{\text{in}}).
\end{multline*}
As $K \subset \Gamma \times \R^d$, we have
\[
\iint_{\Gamma \times \R^d} \left( \int_0^T \indic_{\Omega} (X(t;x,\xi)) \, {\text{d}} t \right)  f^{\text{in}}(x, \xi) \, {\text{d}} x \, {\text{d}} \xi \geq \left(\inf_{(x,\xi) \in K} \int_0^T \indic_{\Omega} (X(t;x,\xi)) \, {\text{d}} t  \right)\iint_{K} f^{\text{in}}(x, \xi) \, {\text{d}} x \, {\text{d}} \xi,
\]
and we therefore obtain, using the Liouville observability result (Lemma~\ref{lem:liouvilleobs}),
\begin{multline}\label{eq:intermedproof1}
\int_0^T \underline{\text{Tr}} \left(\indic_{\Omega_\delta^\Ell} \,  R(t) \, \indic_{\Omega_\delta^\Ell}\right) \, {\text{d}} t \geq C_{{\text{GC}}}[T,K,\Omega] \iint_{(x,\xi) \in K}   f^{\text{in}}(x,\xi) \, {\text{d}} x \, {\text{d}} \xi \\
    -{\sqrt{\frac{2 \gamma_+}{\gamma_-}}} \left(\frac{1/\delta}{\lambda} \right) \left( \frac{\exp  \left(\eta_{\lambda,V} \, T \right) - 1}{\eta_{\lambda,V}} \right)  E_{\hbar,\lambda} (f^{\text{in}}, \, R^{\text{in}}).
\end{multline}
Let us denote by $(R^{\rm in}_k)_{k\in \Gamma^*}$ the Bloch transform of $R^{\rm in}$.

\medskip

$\bullet$ \emph{Töplitz case.} Let us assume $R^{\rm in} = T_\Ell [f^{\rm in}]$. Applying Proposition~\ref{prop:boundToplitz} to~\eqref{eq:intermedproof1}, we obtain
\begin{multline*}
    \int_0^T \underline{{\text{Tr}}} \left(\indic_{\Omega_\delta^\Ell} \,  R(t) \, \indic_{\Omega_\delta^\Ell}\right) \, {\text{d}} t \geq C_{{\text{GC}}}[T,K,\Omega] \iint_{(x,\xi) \in K}   f^{\text{in}}(x,\xi) \, \text{d} x \, \text{d} \xi \\
    -{\sqrt{\frac{2 \gamma_+}{\gamma_-}}} \left(\frac{1/\delta}{\lambda} \right) \left( \frac{\exp  \left(\eta_{\lambda,V} \, T \right) - 1}{\eta_{\lambda,V}} \right)  \sqrt{\frac{1+\lambda^2}{2} \,  d \, \hbar}.
\end{multline*}
Taking $\lambda = \mathrm{Lip}(\nabla V)$ yields~\eqref{eq:theoremToplitz_Obs}.

\medskip

$\bullet$ \emph{Pure state case.} Now we suppose the existence of $(u^{\rm in}_k)_{k \in \Gamma^*} \in L^2_{\rm per}(\Gamma)^{\Gamma^*}$ such that $R^{\rm in}_k = \proj{u^{\rm in}_k}$. Let $f^{\rm in} = \husimi[R^{\rm in}]$, the $\Ell$-periodic Husimi transform of $R$ defined in~\eqref{eqdef:husimi1}--\eqref{eqdef:husimi2}. Taking $\lambda = 1$ and using Equation~\eqref{eq:EhbarHusimi} in Proposition~\ref{prop:EhbarHusimi}, Equation~\eqref{eq:intermedproof1} yields in this case
\begin{multline*}
   \int_0^T \underline{{\text{Tr}}} \left(\indic_{\Omega_\delta^\Ell} \,  R(t) \, \indic_{\Omega_\delta^\Ell}\right) \, {\text{d}} t \geq C_{{\text{GC}}}[T,K,\Omega] \iint_{(x,\xi) \in K}   f^{\text{in}}(x,\xi) \, \text{d} x \, \text{d} \xi \\
    -{\sqrt{\frac{2 \gamma_+}{\gamma_-}}} \left(\frac{1}{\delta} \right) \left( \frac{\exp  \left(\eta_{1,V} \, T \right) - 1}{\eta_{1,V}} \right) \sqrt{d \hbar \, \mathbf{c}_{u^{\rm in}}  + 2 \Delta^2_\Gamma (R^{\rm in})},
\end{multline*}
which is the announced~\eqref{eq:theoremPure_Obs}.
\end{proof}

\addtocontents{toc}{\protect\setcounter{tocdepth}{-1}} 
\section*{Acknowledgements}
We ackowledge the financial support of European Research Council (ERC) under the European Union's Horizon 2020 Research and Innovation Programme – Grant Agreement n°101077204 HighLEAP, awarded to VE. We greatly thank François Golse for having pointed out a subtle difficulty we had first missed. We also kindly thank Thierry Paul and Cyril Letrouit for their discussions. TB kindly thanks Diego Fiorletta for their discussion and for providing references.

\bibliographystyle{plain}
\bibliography{biblio}

@article{golse2022quantitative,
  title={Quantitative observability for the {S}chr{\"o}dinger and {H}eisenberg equations: An optimal transport approach},
  author={Golse, F. and Paul, T.},
  journal={Mathematical Models and Methods in Applied Sciences},
  volume={32},
  number={05},
  pages={941--963},
  year={2022},
  publisher={World Scientific}
}

@article{laurent2013survey,
  title={Internal control of the Schr\" odinger equation},
  author={Laurent, C.},
  journal={Math. Control Relat. Fields},
  year={2014},
volume = {4},
number = {2},
pages = {161-186}
}

@article{gilfeather1973functional,
  title={On a functional calculus for decomposable operators and applications to normal, operator-valued functions},
  author={Gilfeather, F.},
  journal={Transactions of the American Mathematical Society},
  volume={176},
  pages={369--383},
  year={1973}
}

@book{lewin2024spectral,
  title={Spectral theory and quantum mechanics},
  author={Lewin, M.},
  year={2024},
  publisher={Springer}
}

@article{frank2020periodic,
  title={The periodic Lieb--Thirring inequality},
  author={Frank, R.L. and Gontier, D. and Lewin, M.},
  journal={EMS Series of Congress Reports},
  pages={135--154},
  year={2022},
  doi={10.4171/ECR/18-1/8}
}

@book{reed1980methods,
  title={Methods of modern mathematical physics: Analysis of Operators},
  author={Reed, M. and Simon, B.},
  volume={4},
  year={1978},
  publisher={Elsevier}
}

@article{batistic2019statistical,
  title={Statistical properties of the localization measure of chaotic eigenstates and the spectral statistics in a mixed-type billiard},
  author={Batisti{\'c}, B. and Lozej, {\v{C}}. and Robnik, M.},
  journal={Phys. Rev. E},
  volume={100},
  number={6},
  pages={062208},
  year={2019},
  publisher={APS}
}

@inproceedings{zuazua2002remarks,
  title={Remarks on the controllability of the {S}chr{\"o}dinger equation},
  author={Zuazua, E.},
  booktitle={CRM Workshop},
  pages={193--211},
  year={2002}
}

@article{d2025dynamical,
  title={A dynamical {A}mrein-{B}erthier uncertainty principle},
  author={D'Ancona, P. and Fiorletta, D.},
  journal={arXiv preprint arXiv:2504.13746},
  year={2025}
}

@book{burq1993controle,
  title={Contr{\^o}le de l'{\'e}quation des plaques en pr{\'e}sence d'obstacles strictement convexes},
  author={Burq, N.},
  year={1993},
  publisher={Soci{\'e}t{\'e} math{\'e}matique de France}
}

@article{jaffard1990controle,
  title={Contr{\^o}le interne exact des vibrations d'une plaque rectangulaire},
  author={Jaffard, S.},
  journal={Portugaliae mathematica},
  volume={47},
  pages={423--429},
  year={1990}
}

@article{lions1988exact,
  title={Exact controllability, stabilization and perturbations for distributed systems},
  author={Lions, J-L.},
  journal={SIAM review},
  volume={30},
  number={1},
  pages={1--68},
  year={1988},
  publisher={SIAM}
}

@article{lebeau1992controle,
  title={Contr{\^o}le de l'{\'e}quation de {S}chr{\"o}dinger},
  author={L., Gilles},
  journal={Journal de Math{\'e}matiques Pures et Appliqu{\'e}es},
  volume={71},
  number={3},
  pages={267--291},
  year={1992}
}

@article{bardos1992sharp,
  title={Sharp sufficient conditions for the observation, control, and stabilization of waves from the boundary},
  author={Bardos, C. and Lebeau, G. and Rauch, J.},
  journal={SIAM J. Control Optim.},
  volume={30},
  number={5},
  pages={1024--1065},
  year={1992},
  publisher={SIAM}
}

\addtocontents{toc}{\protect\setcounter{tocdepth}{1}} 

\appendix

\section{Various tools}
\subsection{Tools on the Bloch transform} \label{sec:Bloch}
\noindent Recall the notation $\Hf = L^2(\R^d; \, \mathbb{C})$.
\begin{definition}
We define the \emph{Bloch transform} as the bounded operator
\[
\B : \Hf \longrightarrow L^2_{qp}(\Gamma^*, L^2_{\rm per}(\Gamma)),
\]
where
\[
L^2_{qp}(\Gamma^*, L^2_{\rm per}(\Gamma)) := \left\{ u \in L^2_{loc}(\R^d;L^2_{\rm per}(\Gamma)) \; | \; \forall (k,K) \in \Gamma^* \times \Ell^*, \, u(k+K) = U_K \, u(k) \right\},
\]
and $U_K$ is the unitary operator on $L^2_{\rm per}(\Gamma)$ of multiplication by the function $x \mapsto e^{-i K \cdot x}$, such that
\begin{equation} \label{eqdef:BlochTransform}
\forall \phi \in \C^{\infty}_c(\R^d), \quad \; (\B \phi)(k)(x) = \sum_{\ell \in \Ell} \phi(x+\ell) \, e^{-i k \cdot (x+\ell)}, \qquad \forall \, (k,x) \in \Gamma^* \times \Gamma.
\end{equation}
It satisfies in particular
\begin{align}
    &\forall u \in L^2(\R^d), \quad \|u\|^2_{L^2} = \int_{\Gamma^*} \|(\B u)(k)\|^2_{L^2_{\rm per}(\Gamma)} \, \dd k, \label{eq:BlochNormsEq} \\
    &\forall u \in L^2(\R^d), \quad \; u(x) = \fint_{\Gamma^*} (\B u)(k)(x) \, e^{i k \cdot x} \, \dd k, \qquad \text{for a.e. } x \in \R^d. \label{eq:inverseBlochTfformula}
\end{align}
\end{definition}

\begin{proposition} \label{prop:blochTfunbd} For any operator $R \in \mathcal{S}_{\Ell}$, there exists a unique (almost everywhere in $k$) family of self-adjoint operators $\{R_k\}_{k \in \Gamma^*}$ over $L^2_{\rm per}(\Gamma)$ such that
\begin{equation} \label{eq:defBlochTfOp}
    \forall f \in \mathcal{D}(R), \qquad (R f)_k = R_k f_k.
\end{equation}
\end{proposition}
\begin{proof}
The proof is derived from one of lecture notes of Eric Cancès (which are not accessible online) which is stated for . Let $u$ and $v \in \mathcal{D}(R)$. For any $\ell \in \Ell$, we have
\[
\langle v | \tau_\ell R u \rangle_{L^2} = \langle \tau_{-\ell}  v |R u \rangle_{L^2}  = \fint_{\Gamma^*} \langle e^{i k\cdot \ell} \, v_k | (R u)_k \rangle_{L^2_{\rm per}} \, \dd k = \fint_{\Gamma^*} \langle  \, v_k | (R u)_k \rangle_{L^2_{\rm per}}\, e^{-i k\cdot \ell} \, \dd k,
\]
as well as
\[
\langle v | R\tau_\ell u \rangle_{L^2} = \langle R  v |\tau_\ell u \rangle_{L^2}  = \fint_{\Gamma^*} \langle (R v)_k | e^{-i k\cdot \ell} u_k  \rangle_{L^2_{\rm per}} \, \dd k = \fint_{\Gamma^*} \langle (R v)_k |  u_k  \rangle_{L^2_{\rm per}}\, e^{-i k\cdot \ell} \, \dd k.
\]
Up to normalization, $\langle v | \tau_\ell R u \rangle_{L^2}$ and $\langle v | R\tau_\ell u \rangle_{L^2}$ therefore correspond to Fourier coefficients of the $\Gamma^*$-periodic functions $k \mapsto  \langle  \, v_k | (R u)_k \rangle_{L^2_{\rm per}}$ and $k \mapsto \langle (R v)_k |  u_k  \rangle_{L^2_{\rm per}}$. Since $\tau_\ell$ and $R$ commute, the last functions thus have same Fourier coefficients, hence are equal almost everywhere. As such, for almost every $k \in \Gamma^*$ and $u,v \in \mathcal{D}(R)$,
\[
\langle  \, v_k | (R u)_k \rangle_{L^2_{\rm per}} = \langle (R v)_k |  u_k  \rangle_{L^2_{\rm per}}.
\]
Since above is linear in $u_k$ and $v_k$, depends on $v$ only through $v_k$ (left term) and on $u$ only through $u_k$ (right term), there is a self-adjoint (due to the above equality) operator over $L^2_{\rm per}(\Gamma)$, denoted $R_k$, such that
\[
\langle  \, v_k | (R u)_k \rangle_{L^2_{\rm per}} = \langle (R v)_k |  u_k  \rangle_{L^2_{\rm per}} = \langle v_k | R_k | u_k  \rangle_{L^2_{\rm per}}.
\]
The uniqueness of $R_k$ is ensured by the density of the domain of $R$ in $L^2(\R^d)$.
\end{proof}

\begin{corollary} \textbf{Operator composition compatibility of the Bloch transform} \label{cor:fundamentalPropertyBlochOp}
    For any finite family of decomposable operators $(R_n)_{1 \leq n \leq N} \in (\ScEll)^N$, $N \in \N^*$, such that for any $n$, we have
    \begin{equation} \label{eq:fundamentalPropertyBlochOp}
       \forall k \in \Gamma^*, \qquad  (R_N \, R_{N-1} \, \dots \, R_1)_k = R_{N,k} \, R_{N-1,k} \, \dots \, R_{1,k},
    \end{equation}
    where the notation $A_{\bullet}$ stands for the Bloch transform of $A$ mentioned in Proposition~\ref{prop:blochTfunbd}.
\end{corollary}

\begin{proof}
Let $f \in \mathcal{D}(R_N  \, \dots \, R_1)$. Note in particular that for any $n \in \llbracket 1, N-1 \rrbracket$ we have $R_n \, \dots \, R_1 f \in \mathcal{D}(R_{n+1})$. From Equation~\eqref{eq:defBlochTfOp} in Proposition~\ref{prop:blochTfunbd}, we have for any $k \in \Gamma^*$ that
\[
(R_N  \dots  R_1)_k f_k = (R_N  \dots  R_1 f)_k = R_{N,k} (R_{N-1} \dots R_1 f)_k = \dots = R_{N,k}  \dots R_{1,k} f_k,  
\]
proving~\eqref{eq:fundamentalPropertyBlochOp} by uniqueness of the Bloch decomposition (see Proposition~\ref{prop:blochTfunbd}).
\end{proof}

\subsection{Some lemmas} \label{app:lemmas}

\begin{lemma}[\textbf{Periodic-potential Liouville flow change of variable formula}]\label{lem:LiouvilleChangeVar}
Let $V \in \mathcal{C}^{1,1}(\R^d) \cap H^2_{\rm per}(\Gamma)$, and consider, as in the text, the Liouville flow $\Phi$ associated with the ODE system~\eqref{eq:trajectoryInit}. Then for any $t \in \R$ and $\Ell$-periodic-in-the-$x$-variable function $g \in L^1_{x \text{-}\rm per}(\Gamma \times \R^d)$, we have
\begin{equation}
    \label{eq:LiouvilleChangeVar}
    \iint_{\Gamma \times \R^d} g(x,\xi) \, \dd x \, \dd \xi = \iint_{\Gamma \times \R^d} g \left(\Phi_t(x,\xi) \right) \, \dd x \, \dd \xi.
\end{equation}
\end{lemma}

\begin{proof}
    We write
    \[
    \iint_{\Gamma \times \R^d} g(x,\xi) \, \dd x \, \dd \xi = \iint_{\R^d \times \R^d} \, \indic_{x \in \Gamma} \; g(x,\xi) \, \dd x \, \dd \xi.
    \]
    By Liouville's invariance theorem, the flow $\Phi_t$ is a $\mathcal{C}^1$-diffeomorphism of $\R^d \times \R^d$ that leaves $\dd x \, \dd \xi$ invariant, hence, denoting $(X_t(x,\xi),\Xi_t(x,\xi)) = \Phi_t(x,\xi)$, we obtain
    \[
    \iint_{\R^d \times \R^d} \, \indic_{x \in \Gamma} \; g(x,\xi) \, \dd x \, \dd \xi = \iint_{\R^d \times \R^d} \, \indic_{X_t(x,\xi) \in \Gamma} \; g(X_t(x,\xi),\Xi_t(x,\xi)) \, \dd x \, \dd \xi.
    \]
    Decomposing the integral in $x$ on the lattice $\Ell$, the above may be recast as
    \[
    \sum_{\ell \in \Ell} \iint_{\Gamma + \ell \times \R^d} \, \indic_{X_t(x,\xi) \in \Gamma} \; g(X_t(x,\xi),\Xi_t(x,\xi)) \, \dd x \, \dd \xi,
    \]
    that is
\[
\sum_{\ell \in \Ell} \iint_{\Gamma\times \R^d} \, \indic_{X_t(x+\ell,\xi) \in \Gamma} \; g(X_t(x+\ell,\xi),\Xi_t(x,\xi)) \, \dd x \, \dd \xi.
\]
Now recall~\eqref{eq:pseudoPeriodicityPhi}, stating that the $\Ell$-periodicity of $V$ implies that $X_t(x+\ell,\xi) = X_t(x,\xi) + \ell$. The above becomes, using the $\Ell$-periodicity of $g$ in its first variable,
\[
\sum_{\ell \in \Ell} \iint_{\Gamma\times \R^d} \, \indic_{X_t(x,\xi) \in \Gamma - \ell} \; g(X_t(x,\xi),\Xi_t(x,\xi)) \, \dd x \, \dd \xi,
\]
which rewrites
\[
 \iint_{\Gamma\times \R^d} \left(\sum_{\ell \in \Ell} \indic_{X_t(x,\xi) \in \Gamma - \ell} \right) g(X_t(x,\xi),\Xi_t(x,\xi)) \, \dd x \, \dd \xi.
\]
Since $\{\Gamma - \ell\}_{\ell \in \Ell}$ is a partition of $\R^d$, it holds for all $(x,\xi) \in \Gamma \times \R^d$ that
\[
\sum_{\ell \in \Ell} \indic_{X_t(x,\xi) \in \Gamma - \ell} = 1,
\]
allowing to conclude.
\end{proof}

\noindent The following lemma provides the Bloch decomposition of the flow operator $U(t)$ and the transportation cost $c_\lambda (x,\xi)$.
\begin{lemma} [\textbf{Bloch transform of the evolution and cost operators}]
\label{lem:BlochTransFormulas}
For almost any $(t,x,\xi) \in \R \times \R^d \times \R^d$ and $k \in \Gamma^*$, we have
\begin{align}
    (U(t))_k &= U_k(t), \label{eq:U_t_k} \\
\left( c_\lambda (x, \xi) \right)_k &= \hat{c}_\lambda (x,\xi - \hbar k), \label{eq:c_lba_k}
\end{align}
where $(R)_k$ stands for the $k^{th}$ component of the Bloch transform of an operator $R$, as defined in Appendix~\ref{sec:Bloch}. We recall that the flow operators $U(t)$ (on $L^2(\R^d)$) and $U_k(t)$ (on $L^2_{\rm per}(\Gamma)$) are associated to the von Neumann equations related to the quantum Hamiltonians $\HH$ and $\HH_k$ defined in~\eqref{eq:quantumHamiltonian}--\eqref{eq:quantumHamiltonian_k} and the costs $c_\lambda$ and $\hat{c}_\lambda$ are defined in~\eqref{eqdef:transpcost}.
\end{lemma}

\begin{proof} 
$\bullet$ Equation~\eqref{eq:U_t_k} corresponds to a time-dependent version of the Bloch Theorem. It is given in~\cite{gilfeather1973functional} that if an operator $A$ is decomposable as $\int^\oplus_\Lambda A(\lambda) \dd \mu(\lambda)$ on a
direct integral of Hilbert spaces $\mathfrak{H} = \int^\oplus_\Lambda \mathfrak{H}_\lambda \dd \mu(\lambda)$ and $f$ is an analytic function, the operator $f(A)$ is decomposable and $f(A(\lambda)) = f(A)(\lambda)$ almost everywhere. Since the Bloch decomposition can be interpreted in this formalism, substituting $A$ by $\HH$ and $\Lambda$ by $\Gamma^*$ (see Reed and Simon~\cite{reed1980methods}) and since for any $t \in \R$, the application $h \in \mathbb{R} \mapsto \exp(i t \, h / \hbar )$ is analytic, we deduce that for all $t \in \R$, we have for almost any $k \in \Gamma^*$ that
\[
\left(\exp(i t \, \HH / \hbar ) \right)_k = \exp(i t \, \HH_k / \hbar ),
\]
that is~\eqref{eq:U_t_k}.

    \medskip

    $\bullet$ Let us now prove~\eqref{eq:c_lba_k}. Consider $\phi \in \mathcal{C}^{\infty}_c(\R^d)$. On the one hand, for any $k \in \Gamma^*$, we have
    \[
    \phi_k(y) = \sum_{\ell \in \Ell} \phi(y+\ell) \, e^{-i k \cdot (y+\ell)}, \qquad \forall \, y \in \Gamma,
    \]
    and on the other hand,
    \[
    (c_\lambda (x,\xi) \phi)_k(y) = \left(\lambda^2 \Theta(|P_\Gamma(x-y)|^2) + (\xi - \hbar k + i\hbar \nabla y)^2 \right)  \phi_k(y) 
    \]
    as
    \[
   \left( c_\lambda (x,\xi) \phi \right) (y) = \left(\lambda^2 \Theta(|P_\Gamma(x-y)|^2) + (\xi + i \hbar \nabla_y)^2 \right) \phi (y).
    \]
    Indeed, the Bloch transform commutes with the multiplication with an $\Ell$-periodic function.
\end{proof}

\begin{lemma} [\textbf{Commutator formulas}] \label{lem:commutators} Consider $V \in H^2_{\rm per}(\Gamma)$. We have the identities, seeing all following operators as acting on $L^2_{\rm per}(\Gamma)$ functions of the variable $y$,
\begin{align}
\frac{i}{\hbar} \left[V(y), \, (\xi + i \hbar \nabla_y)^2 \right] &= ( \xi + i \hbar \nabla_y) \cdot \nabla V(y) + \nabla V(y) \cdot \left( \xi + i \hbar \nabla_y \right), \label{eqlem:identity1} \\
\frac{i}{\hbar} \left[V(y), \,\lambda^2 \Theta(|P_\Gamma(x-y)|^2) \right] &= 0, \label{eqlem:identity2}  \\
\frac{i}{\hbar} \left[\frac{1}{2}(-i \hbar\nabla_y + \hbar k)^2, \, (\xi + i \hbar \nabla_y)^2 \right] &= 0, \label{eqlem:identity3} \\
\frac{i}{\hbar} \left[\frac{1}{2}(-i \hbar\nabla_y + \hbar k)^2, \, \lambda^2  \Theta(|P_\Gamma(x-y)|^2) \right] &=   \lambda^2 (-i\hbar \nabla_y + \hbar k) \cdot \left[ P_{\Gamma}(y-x) \; \Theta'(|P_\Gamma(x-y)|^2 ) \right] \\
&+ \lambda^2 P_{\Gamma}(y-x) \; \Theta'(|P_\Gamma(x-y)|^2 ) \cdot (-i\hbar \nabla_y + \hbar k).  \label{eqlem:identity4} 
\end{align}
\end{lemma}

\begin{proof}
    Identities~\eqref{eqlem:identity2}--\eqref{eqlem:identity3} are straightforward as the involved operators commute. Moreover, Identity~\eqref{eqlem:identity4} is just a consequence of~\eqref{eqlem:identity1}, replacing $\xi$ by $-\hbar k$ and $V$ by $y \mapsto -\Theta(|P_\Gamma(x-y)|^2)$, as the latter indeed belongs to $H^2_{\rm per}(\Gamma)$. Therefore, we only need to prove~\eqref{eqlem:identity1}. 

    \medskip
    
    \noindent We have
    \[
    (\xi + i \hbar \nabla_y)^2 = |\xi|^2 + 2 \xi \cdot (i \hbar \nabla_y) + ( i \hbar \nabla_y)^2,
    \]
    so that
    \begin{multline*}
    (\xi + i \hbar \nabla_y)^2 V(y) = |\xi|^2 V(y) + 2 \xi \cdot \left(i \hbar \nabla V (y) + V(y) (i \hbar \nabla_y) \right) \\+ \left( -\hbar^2 \Delta_y V(y) + 2i\hbar \nabla V(y) (i \hbar \nabla_y) + V(y) (i \hbar \nabla_y)^2 \right).
    \end{multline*}
    On the other hand,
    \begin{equation*}
   V(y) (\xi + i \hbar \nabla_y)^2  = V(y)\left(  |\xi|^2 + 2 \xi \cdot ( i \hbar \nabla_y) + (i \hbar \nabla_y)^2 \right),
    \end{equation*}
    so that
    \begin{equation}
\frac{i}{\hbar}\left[V(y), \, (\xi + i \hbar \nabla_y)^2 \right]  =  2 \xi \cdot \nabla V(y)  + i \hbar \Delta_y V(y) + 2 \nabla V(y) (i\hbar \nabla_y).
    \end{equation}
We now compute
\[
( \xi + i \hbar \nabla_y) \cdot \nabla V(y) = \xi  \cdot \nabla V(y) + i \hbar \Delta_y V(y) + \nabla V(y) (i\hbar \nabla_y),
\]
showing that
\[
\frac{i}{\hbar}\left[V(y), \, (\xi + i \hbar \nabla_y)^2 \right]  = ( \xi + i \hbar \nabla_y) \cdot \nabla V(y) + \nabla V(y) \cdot \left( \xi + i \hbar \nabla_y \right).
\]
\end{proof}

\end{document}